\documentclass{article}
\usepackage[utf8]{inputenc}

\usepackage{comment}

\usepackage{graphicx} 
\usepackage{amsmath,amsthm,verbatim,amssymb,amsfonts,amscd, graphicx}
\usepackage[utf8]{inputenc}
\usepackage{graphics}
\usepackage{hyperref}
\usepackage[baseline]{euflag}
\usepackage{enumitem}
\usepackage{apptools}
\usepackage{titlesec}
\usepackage{appendix}
\usepackage{xcolor}
\usepackage{dsfont}
\AtAppendix{\counterwithin{lemma}{section}}

\theoremstyle{plain}

\newtheorem{theorem}{Theorem}
\newtheorem{corollary}[theorem]{Corollary}
\newtheorem{lemma}[theorem]{Lemma}
\newtheorem{proposition}[theorem]{Proposition}

\newtheorem{question}[theorem]{Question} 
\newtheorem{remark}[theorem]{Remark}

\theoremstyle{definition}
\newtheorem{definition}[theorem]{Definition}

\newcommand{\RL}{\mathbb{R}}

\newcommand{\Cc} { \mathcal{C}(c_1,c_2,c_3)}

\newcommand{\Z}{\mathbb{Z}}
\newcommand{\R}{\mathbb{R}}

\newcommand{\Cov}{\mathrm{Cov}}

\newcommand{\fx} {
\lfloor x_0 \rfloor
}

\newcommand{\nm} {
N_{\max}
}

\newcommand{\cptwo} {
C^{\prime}_2
}
\newcommand{\Un} {
U^{(n)}
}

\newcommand{\ctn}{
c_2^{\prime}(d,n)
}
\newcommand{\ctno}{
c_2^{\prime}(d,n+1)
}
\usepackage{authblk}

\begin{document}

\title{On the monotonicity of discrete entropy for log-concave random vectors on $\mathbb{Z}^d$}
\author[1]{Matthieu Fradelizi}
\author[1]{Lampros Gavalakis\thanks{L.G. has received funding from the European Union's Horizon 2020 research and innovation program
 under the Marie Sklodowska-Curie grant agreement No 101034255 {\large \euflag} and by the B{\'e}zout Labex, funded by ANR, reference ANR-10-LABX-58.}}
 \author[1]{Martin Rapaport}
 \affil[1]{Univ Gustave Eiffel, Univ Paris Est Creteil, CNRS, LAMA UMR8050 F-77447 Marne-la-Vall{\'e}e, France}

\maketitle

\abstract{ We prove the following type of discrete entropy monotonicity for
sums of isotropic, log-concave, independent and identically distributed random vectors $X_1,\dots,X_{n+1}$ on $\mathbb{Z}^d$:
$$
H(X_1+\cdots+X_{n+1}) \geq  H(X_1+\cdots+X_{n}) + \frac{d}{2}\log{\Bigl(\frac{n+1}{n}\Bigr)} +o(1),
$$
where $o(1)$ vanishes as $H(X_1) \to \infty$. Moreover, for the $o(1)$-term, we obtain a rate of convergence $ O\Bigl({H(X_1)}{e^{-\frac{1}{d}H(X_1)}}\Bigr)$, where the implied constants depend on $d$ and $n$. This generalizes to $\mathbb{Z}^d$ the one-dimensional result of the second named author (2023). As in dimension one, our strategy is to establish that the discrete entropy $H(X_1+\cdots+X_{n})$ is close to the differential (continuous) entropy $h(X_1+U_1+\cdots+X_{n}+U_{n})$, where $U_1,\dots, U_n$ are independent and identically distributed uniform random vectors on $[0,1]^d$ and to apply the theorem of Artstein, Ball, Barthe and Naor (2004) on the monotonicity of differential entropy. In fact, we show this result under more general assumptions than log-concavity, which are preserved up to constants under convolution. Namely, we consider families of random variables for which, as the determinant of the covariance matrix increases, the probability mass function {\bf i)} is bounded above in terms of the the determinant of the covariance matrix, {\bf ii)} has subexponential tails, {\bf iii)} has (discrete) bounded variation. In order to show that log-concave distributions satisfy our assumptions in dimension $d\ge2$, more involved tools from convex geometry are needed because a suitable position is required. We show that, for a log-concave function on $\R^d$ in isotropic position, its integral, barycenter and covariance matrix are close to their discrete counterparts. Moreover, in the log-concave case, we weaken the isotropicity assumption to what we call \textit{almost isotropicity}. One of our technical tools is a discrete analogue to the upper bound on the isotropic constant of a log-concave function, which extends to dimensions $d\ge1$ a result of Bobkov, Marsiglietti and Melbourne (2022) in dimension one and may be of independent interest. 
}\\


\medskip


\section{Introduction}
\subsection{Monotonicity of entropy}

The {\it differential entropy} of an $\R^{d}$-valued random vector $X$ with density $f$ is defined as
$$
h(X)=-\int_{\RL^d} f(x) \log f(x) dx ,
$$
if the integral exists. When $X$ is supported on a strictly lower-dimensional
set than $\R^d$ (and hence does not have a density with respect to the
$d$-dimensional Lebesgue measure), one sets $h(X) =-\infty$. The {\it entropy power} of $X$ is defined by
$N(X)=e^{2h(X)/d}$. 

Let $X_1,\ldots,X_n$ be i.i.d. continuous random variables. A generalization of the classical Entropy Power Inequality (EPI) of Shannon \cite{shannon1948mathematical} and Stam \cite{stam}, due to Artstein, Ball, Barthe and Naor, \cite{artstein} states that

\begin{equation*}
    h\Bigl(\frac{1}{\sqrt{n+1}}\sum_{i=1}^{n+1}X_{i}\Bigr)\geq h\Bigl(\frac{1}{\sqrt{n}}\sum_{i=1}^{n}X_{i}\Bigr),
\end{equation*}
which by the scaling property of differential entropy is equivalent to 

\begin{equation} \label{cntsepid}
h\Bigl(\sum_{i=1}^{n+1}{X_i}\Bigr) \geq h\Bigl(\sum_{i=1}^{n}{X_i}\Bigr) + \frac{d}{2}\log{\Bigl(\frac{n+1}{n}\Bigr)}.
\end{equation}
It can also be interpreted as the monotonic increase of entropy along the Central Limit Theorem (CLT) \cite{barronentropy}. Later, other proofs and generalizations were given by Shlyakhtenko \cite{shlyak}, Tulino and Verd\'u \cite{tulino} and  Madiman and Barron \cite{mokshayepi}.  

Let $A$ be a discrete (finite or countable) set and let $X$ be a random variable supported on $A$ with probability mass function (p.m.f.) $p$ on $A$. The discrete (Shannon) entropy of $X$ is 
$$
H(X) = -\sum_{x \in A}{p(x)\log{p(x)}}.
$$
An exact analogue of \eqref{cntsepid} cannot be true for discrete random variables as can be seen by taking $n=1, d=1$ and considering deterministic or even close to deterministic random variables. Nevertheless, Tao \cite{tao_sumset_entropy}, using ideas from additive combinatorics, proved that for any independent and identically distributed random variables $X_1,X_2$ taking values in a finite subset of a torsion-free (abelian) group
\begin{equation} \label{taoEPI}
H(X_1+X_2) \geq H(X_1) + \frac{1}{2}\log{2} - o(1),
\end{equation}
where $o(1) \to 0$ as $H(X_1) \to \infty$. This can be seen as a discrete analogue of \eqref{cntsepid} for $d=n=1$.
Besides \eqref{taoEPI}, other discrete versions of the EPI have been studied by several authors. In \cite{adaptivesensing,abbe}, a discrete EPI similar to \eqref{taoEPI} was proved, with a worse constant than $\frac{1}{2}{\log{2}}$, but with sharp quantitative bounds for the $o(1)$-term (see also \cite{logconcaveisit} for a detailed discussion on connections with that work). An EPI for the binomial family was proved in \cite[Theorem 1.1]{harremoesepibinomial}; in \cite[Theorem 4.6]{madimancyclic} it was shown that   
$e^{2H(X)} -1$
is superadditive with respect to convolution on the class of uniform distributions on finite subsets of the integers (see also \cite{madimanbernoulli} for extensions to R{\'e}nyi entropy and \cite{madimanmajorization} for dimensional extensions). Finally the constant $\frac{1}{2}\log{2}$ was shown to be improvable for the class of uniform distributions in \cite{woo2015discrete}.

In addition to \eqref{taoEPI}, Tao conjectured that for any independent and identically distributed random variables $X_1,\ldots,X_{n+1}$ taking values in a finite subset of a torsion-free group
\begin{equation} \label{Taoconjecture}
H\Bigl(X_1 + \cdots +X_{n+1}\Bigr) \geq H\Bigl(X_1 + \cdots +X_{n}\Bigr) + \frac{1}{2}\log{\Bigl(\frac{n+1}{n}\Bigr)} - o(1),
\end{equation}
as $H(X_1) \to \infty$ depending on $n$. A special case of this conjecture was recently proven in \cite{gavalakis}, where it is shown that \eqref{Taoconjecture} is satisfied by log-concave random variables on the integers with explicit rate for the $o(1)$ that is exponential in $H(X_1)$:
\begin{theorem}[\cite{gavalakis}] \label{onedTh}
Let $X_1,\ldots,X_n$ be i.i.d. log-concave random variables on $\mathbb{Z}$. Then
$$
H\Biggl(\sum_{i=1}^{n+1}{X_i}\Biggr) \geq H\Biggl(\sum_{i=1}^{n}{X_i}\Biggr) + \frac{1}{2}\log{\Bigl(\frac{n+1}{n}\Bigr)} - O\Bigl(H(X_1)e^{-H(X_1)}\Bigr),
$$
as $H(X_1) \to \infty$.
\end{theorem}

This statement can be interpreted as a type of ``approximate" monotonicity in the {\em discrete} entropic CLT \cite{discreteCLT}. We recall here that an integer valued random variable $X$ with p.m.f. $p$ is said to be {\it log-concave} if 
\begin{equation} \label{1dlogconcavedef}
    p(k)^2 \geq p(k-1)p(k+1),
\end{equation}
for every $k \in \mathbb{Z}.$

\subsection{Main results and paper outline}

Our goal is to extend Theorem \ref{onedTh} to higher dimensions and to the following larger class of probability distributions:

\begin{definition} \label{familydef}
Let $\Cc$ be the class of p.m.f.s $p$ on $\mathbb{Z}^d$ for which, denoting $\sigma := \det{\Bigl(\Cov{(p)}}\Bigr)^{\frac{1}{2d}}$,
\begin{enumerate} 
\item \label{A1} $$\max_{k \in \mathbb{Z}^d}{p(k)} \leq \frac{c_1}{\sigma^d},$$ 

\item \label{A2} for all $k \in \mathbb{Z}^d$ with $\|k\|_{\infty} \geq c_2 \sigma$, 
$$
p(k) \leq \frac{c_2}{\sigma^d}e^{-\frac{\|k\|_1}{c_2\sigma}} \quad \text{ and}
$$

\item \label{A3} $$\sum_{i=1}^d{\sum_{k \in \mathbb{Z}^d}}{|p(k)-p(k-e_i)|}\leq \frac{c_3}{\sigma}.$$
\end{enumerate}

\end{definition}

Our first main result is the following one: 
\begin{theorem} \label{EPI}
For any i.i.d. random vectors $X_1,\ldots,X_{n+1}$ on $\mathbb{Z}^d$ such that their common p.m.f. belongs to the class $\mathcal{C}(c_1,c_2,c_3)$, for some absolute constants $c_1, c_2, c_3$ that depend on the dimension only, 
\begin{equation} \label{eq:EPI}
H(X_1+\cdots+X_{n+1}) \geq  H(X_1+\cdots+X_{n}) + \frac{d}{2}\log{\Bigl(\frac{n+1}{n}\Bigr)} - o(1)
\end{equation}
as $H(X_1) \to \infty$. Moreover, we have the rate of convergence $o(1) = O_{d,n}\Bigl({H(X_1)}{e^{-\frac{1}{d}H(X_1)}}\Bigr)$.  In particular, \eqref{eq:EPI} holds for any i.i.d. log-concave random vectors $X_1,\ldots,X_{n+1}$ on $\mathbb{Z}^d$ with almost isotropic extension. 
\end{theorem}
{

It may be observed that \eqref{eq:EPI} is stronger than the original conjecture of Tao by a factor of $d$ in front of the logarithmic term. This is because of the almost isotropicity assumption which imposes that the random variables are really $d$-dimensional. In fact, without this assumption we still have the original statement under log-concavity as shown by the corollary below. Clearly, we may  not get the constant $d$ without such a dimensionality assumption, since we may always embed a one-dimensional vector in $\mathbb{Z}^d$ for any $d>1.$ 
\begin{corollary}
For any i.i.d. log-concave random vectors $X_1,\ldots,X_{n+1}$ on $\mathbb{Z}^d$
\begin{equation} 
H(X_1+\cdots+X_{n+1}) \geq  H(X_1+\cdots+X_{n}) + \frac{1}{2}\log{\Bigl(\frac{n+1}{n}\Bigr)} - o(1).
\end{equation}
Moreover, we have the rate of convergence $o(1) = O_{d,n}\Bigl({H(X_1)}{e^{-\frac{1}{d}H(X_1)}}\Bigr)$.
\end{corollary}
\begin{proof}
Note that
\begin{equation} \label{dataproc}
H(X_1+\cdots+X_{n+1}) -  H(X_1+\cdots+X_{n}) = I(X_1+\cdots+X_{n+1};X_{n+1}) \geq I(g(X_1+\cdots+X_{n+1});g(X_{n+1})),
\end{equation}
where $I(X;Y) = H(X) - H(X|Y)$ denotes the mutual information and the inequality holds for any function $g$ by the data processing inequality. 

Take $g:\mathbb{Z}^d\to \mathbb{Z}$ to be the first-coordinate map. It can be verified that if $X$ belongs to the class $\mathcal{C}$ for some constants, then each coordinate also belongs to the corresponding class in $\mathbb{Z}$ with appropriate constants depending on the initial ones. We may then apply the result for $d=1$ (where almost isotropicity is satisfied trivially) we obtain the result by \eqref{dataproc}.
\end{proof}
}
 In Subsection \ref{defsandnots}, we will explicitly detail the definitions we use for \textit{discrete log-concavity in $\mathbb{Z}^{d}$} as well as the definition of \textit{almost isotropic extension}.  
Theorem \ref{EPI} readily follows from the next theorem together with the generalized EPI in $\mathbb{R}^d$ and can be interpreted as an approximation result of discrete and continuous entropies in higher dimensions: 
\begin{theorem} \label{diffentropyapprox}
For any i.i.d. random vectors $X_1,\ldots,X_{n+1}$ on $\mathbb{Z}^d$ such that their common p.m.f. belongs to the class $\mathcal{C}(c_1,c_2,c_3)$, for some absolute constants $c_1, c_2, c_3$ that depend on the dimension only, 
\begin{equation} \label{diffapproxeq}
h(X_1+\cdots+X_n + U_1+\cdots+U_n) = H(X_1+\cdots+X_n) + o(1),
\end{equation}
where $o(1) \to 0$ as $H(X_1) \to \infty$, $h$ stands for the differential entropy, $H$ denotes the discrete Shannon entropy and $U_1,\ldots,U_n$ are independent continuous uniforms on $[0,1]^d$. Moreover, we have the rate of convergence $o(1) = O_{d,n}\Bigl({H(X_1)}{e^{-\frac{1}{d}H(X_1)}}\Bigr)$. In particular, \eqref{diffapproxeq} holds for any i.i.d. log-concave random vectors $X_1,\ldots,X_{n+1}$ on $\mathbb{Z}^d$ with almost isotropic extension. 

\end{theorem}

\begin{remark}
Let $X$ be a log-concave random vector on $\mathbb{Z}^d$  with almost isotropic extension. Denote the covariance matrix of $X$ by $K_X$, $\sigma^2 := \det{(K_X)}^{\frac{1}{d}}$ and let $U$ be a uniform on the open box $[0,1]^d$. Then 
\begin{equation} \label{gaussmax}
H(X) = h(X+U) \leq \frac{d}{2}\log{\biggl(\det{(K_X+\frac{1}{12}I_{\mathrm{d}})^{\frac{1}{d}}}2 \pi e\biggr)} \simeq \frac{d}{2}\log{\biggl(\sigma^2 2 \pi e\biggr)} ,
\end{equation}
by the fact that the Gaussian maximizes the entropy under fixed covariance matrix, which implies that for any random variable $X$ on $\R^d$ one has 
\begin{equation} \label{gaussmax-entropy}
h(X)\le \frac{d}{2}\log(2\pi e(\det\Cov(X))^{1/d}).
\end{equation}
Thus, $H(X) \to \infty$ implies $\sigma \to \infty$ and therefore it suffices to prove \eqref{diffapproxeq} with $o(1) \to 0$ as $\sigma \to \infty$.
\end{remark}

We prove Theorem \ref{diffentropyapprox} in Section \ref{relaxedsection}. First we use the second property of Definition \ref{familydef}, which is an exponential tail bound, in order to control the \textit{entropy tails}. We use the third property, which is a type of smoothness, to control the error term arising from the approximation of the density of $\sum_i{X_i + U_i}$ by the p.m.f. of $\sum_i{X_i}$. Combined with the Taylor-type estimate of Lemma \ref{elementaryestimatelemma}, which is analogous to an estimate used in \cite{tao_sumset_entropy}, this approximation is the main technical step in the proof of Theorem \ref{diffentropyapprox}. The first property of Definition \ref{familydef} ensures that Lemma \ref{elementaryestimatelemma} is applicable.

In Section \ref{logconcavesection} we will show that the assumption that the p.m.f. of $X_1$ satisfies these three properties is weaker than almost isotropic log-concavity in the sense that any almost isotropic log-concave p.m.f. on $\mathbb{Z}^d$ satisfies these properties with appropriate constants. We remark here that, since our definition of log-concavity reduces to the usual one for the one-dimensional case, Theorems \ref{EPI} and \ref{diffentropyapprox} are also relaxed versions of the corresponding results obtained in \cite{gavalakis}. We prove the following theorem, which allows us to establish that if $p$ is log-concave with almost isotropic extension then $p$ belongs to the class $\mathcal{C}\left(c_1(d),c_2(d),c_3(d)\right)$ for some absolute constants $c_1(d),c_2(d),c_3(d)$ that depend on the dimension only (in particular, this theorem is essential for proving item \ref{logconcave1} in Theorem  \ref{logconcavesatisfyTh}):
\begin{theorem} \label{discreteUB} 
Suppose $p$ is a log-concave p.m.f. on $\mathbb{Z}^d$ with almost isotropic extension and covariance matrix $\Cov{(p)}$. Then there exists an absolute constant $C^{\prime}$ (independent of the dimension), such that
$$
\max_{k \in \mathbb{Z}^d}{p(k)} \leq \frac{C^{\prime}}{\det\Bigl(\Cov_{\mathbb{Z}^d}(p)\Bigr)^{\frac{1}{2}}}
$$
provided that $\det\left(\Cov_{\mathbb{Z}^d}(p)\right)$ is large enough depending on $d$.
\end{theorem}

Theorem \ref{discreteUB} can be seen as $d$-dimensional analogue of the following result, due to Bobkov, Marsiglietti and Melbourne \cite{bobkov2022concentration}, who also studied discrete versions of the EPI (up to multiplicative constants) for R{\'e}nyi entropies of log-concave distributions and which is an important tool in the one-dimensional case:
\begin{theorem} \label{bobkovprop}
\cite[Theorem 1.1]{bobkov2022concentration}
If a random variable $X$ follows a discrete log-concave p.m.f. $f$ on $\mathbb{Z}$, then   
\begin{equation} \label{max-sigma}
 \max_{k \in \mathbb{Z}}f(k) \leq \frac{1}{\sqrt{1+4\sigma^2}}, 
\end{equation}
where $\sigma^2=\mathrm{Var}(X)$. 
\end{theorem}

Theorem \ref{discreteUB} may also be seen as a discrete analogue of a dimensional upper bound on the isotropic constant $L_f$.
The corresponding (continuous) lower bound on $L_f$ can be deduced from \eqref{gaussmax-entropy} and the fact that $e^{-h(X)}\le\max(f)$:
\[
L_f=\left(\max_{\mathbb{R}^d}{f}\right)^{\frac{1}{d}}\mathrm{det}(\mathrm{Cov}(f))^{\frac{1}{2d}}\ge e^{-\frac{h(X)}{d}}\mathrm{det}(\mathrm{Cov}(f))^{\frac{1}{2d}}\ge\frac{1}{\sqrt{2\pi e}}.
\]
It is well known that a dimensional upper bound also exists and we denote by $L_d$ the maximum of the isotropic constants $L_f$ among log-concave functions $f$ in $\R^d$.
The latter is related to the famous hyperplane (or slicing) problem in convex geometry, a longstanding conjecture recently solved in a major breakthrough by Klartag and Lehec~\cite{klartaglehec}, following decisive progress by Guan~\cite{guan}. A second, independent proof was subsequently given by Bizeul~\cite{bizeul}. The hyperplane and isotropic constant theorems, formerly known as conjectures, are in fact equivalent, as shown in~\cite[Theorem 3.1.2]{Giannopoulos14}.\\

\begin{theorem}\cite[Theorem 1.1]{klartaglehec}
\label{klartgalehec}
There exists a universal constant $c > 0$ such that, for any dimension $d$, any convex body $K \subset \R^d$ in isotropic position and of volume $1$, and any direction $\theta \in \mathbb{S}^{d-1}$, one has
\begin{equation} \label{hyperplaneEq}
    |K \cap \theta^{\perp}|_{d-1} \geq c.
\end{equation}
Equivalently, there exists a universal constant $C > 0$ such that
\begin{equation} \label{isotropicEq}
    L_d \leq C.
\end{equation}
\end{theorem}

Henceforth, we will be referring to \eqref{hyperplaneEq} of Theorem \ref{klartgalehec} as the {\em Hyperplane Theorem} and to \eqref{isotropicEq} as the {\em Isotropic Constant Theorem}.

Our method for proving Theorem \ref{discreteUB} is to use the corresponding continuous result. 
To this end, we obtain the approximations 
\begin{align*}
\Bigl|\int_{\mathbb{R}^d}{f} - \sum_{\mathbb{Z}^d}{f}\Bigr| &= o_{d}(1), \quad \text{as } \sigma \to \infty, \\
\Bigl|\int_{\mathbb{R}^{d}}xf - \sum_{k\in \mathbb{Z}^{d}}kf(k)\Bigr|&= O_d(1), \quad \text{as } \sigma \to \infty, \\
\Bigl|\det\Bigl(\Cov_{\mathbb{Z}^d}{(f)}\Bigr)- \det\Bigl(\Cov_{\mathbb{R}^d}{(f)}\Bigr)\Bigr| &= O_d(\sigma^{2d-1}) , \quad \text{as } \sigma \to \infty,
\end{align*}
for any isotropic log-concave density $f$. This is done in Section \ref{cntsdiscretesection}. Before that, after recalling some know facts about Ball's bodies in Section \ref{ballbodiessec}, we prove the exponential tails bound in Section \ref{concentrationlemmasec}. Although our results hold under the more general almost isotropicity assumption, for better illustration of the ideas, we assume first in Sections \ref{ballbodiessec}---\ref{cntsdiscretesection} that the continuous extension $f$ is isotropic. In Section \ref{pisotropicsection} we show how to relax this assumption (see Remark \ref{remarkrelaxassumption}). 

Finally, in Section \ref{concludingsection} we conclude with a brief discussion of our assumptions as well as an open question.

\subsection{Notations and definitions} \label{defsandnots}
In this subsection, we will detail the notations as well as key definitions. \\

\label{secdefs}
{\bf Big- and small-$O$ notation. } Let $f$ be a real-valued function and $g$ another strictly positive function. 
We write $f = O(g)$ if there exist positive absolute constants $N, C$ such that $|f(x)| \leq Cg(x)$ for every $x \geq N$. 
If $N,C$ are absolute up to a parameter $d$, we write $f = O_d(g)$.
Analogously, we write $f = \Omega(g)$, if $|f(x)| \geq Cg(x)$ for every $x \geq N$. 
If $f = \Omega(g)$ and $f = O(g)$, we write $f = \Theta(g)$ (with the analogous definitions for $\Omega_d$ and $\Theta_d$).
When it is more convenient, we will write $f \lesssim_d g$ for $f = O_d(g)$ and $f \simeq_d g$ for $f = \Theta(g)$. We write $f(x) = o(g(x))$ if $\lim_{x\to \infty}\frac{f(x)}{g(x)} = 0$.

\vspace{0.5 cm}
\noindent
{\bf Convex bodies.} A \textit{convex body} is a convex set that is compact and has a non-empty interior. For a convex body $K\in \mathbb{R}^{d}$, one denotes by $|K|$ its volume. The \textit{convex hull} of a set $A$, denoted by $\mathrm{conv(A)}$, is the smallest convex set containing  the set $A$. 

\vspace{0.5 cm}
\noindent
{\bf Log-concavity and convexity in $\mathbb{Z}^d$.  }
For $d>1$, there are more than one definitions of log-concavity that have been used in different contexts. We will use a quite general definition that implies several other notions of log-concavity, the one of \textit{extensible log-concave functions}. The reader is referred to Murota \cite{Murota} for an extensive discussion of discrete convexity in higher dimensions.   
A function $f:\mathbb{Z}^{d}\rightarrow \mathbb{R}\cup \{+\infty\}$ is said to be \textit{convex-extensible} if there exists a convex function $\bar{f}:\mathbb{R}^{d}\rightarrow \mathbb{R}\cup \{+\infty\}$ such that 
\begin{equation*}
    \bar{f}(z)=f(z) \hspace{0.3cm} (\forall z\in \mathbb{Z}^{d}) \hspace{0.1cm} .
\end{equation*}
We refer to $\bar{f}$ as the {\em extension} of $f$.
Based on the above, we define
\begin{definition}[Log-concave extensible functions] \label{LCdef}
A function $f:\mathbb{Z}^{d}\rightarrow \mathbb{R_+}$ is said to be log-concave extensible if there exists $V:\mathbb{Z}^{d}\rightarrow \mathbb{R}\cup \{+\infty\}$ 
convex-extensible such that
\begin{equation*}
    f(z)=e^{-V(z)}.
\end{equation*}
Throughout when we say that $f$ is log-concave, we mean that $f$ is log-concave extensible. We say that a random vector $X$ with values in $\Z^d$ and p.m.f. $p:\Z^d \to [0,1]$ is log-concave, if $p$ is log-concave.

\end{definition}



Similarly, we define a $\mathbb{Z}^{d}$-convex set (commonly referred in the literature as \textit{hole-free} set \cite{Murota}):
\begin{definition}[$\mathbb{Z}^{d}$-convexity]
A set $A\subset \mathbb{Z}^{d}$ is said to be $\mathbb{Z}^{d}$-convex if and only if its indicator function  $\chi_{A}$,  defined by $\chi_{A}(x)=1$ if 
$x\in A$ and $0$ otherwise, is log-concave-extensible or equivalently if 
\begin{equation*}
    A=\mathrm{conv}(A)\cap \mathbb{Z}^{d} \hspace{0.1cm}.
\end{equation*}
Notice that the support of a log-concave function $f$ on $\Z^d$ is $\Z^d$-convex.
\end{definition}

\begin{remark}\label{murota} 
For $d=1$ our definition of log-concavity is equivalent to the usual definition $p(k)^2 \geq p(k-1)p(k+1), k \in \mathbb{Z}$, which is preserved under convolution (e.g. \cite{hoggar}). However, for $d > 1$ log-concavity may not be preserved in general, as pointed out by Murota \cite[Example 3.15]{Murota} by considering two log-concave distributions supported on 
$S_1 = \{(0,0), (1,1)\}$ and $S_2 = \{(0,1),(1,0)\}$ respectively. This is because the Minkowski sum $S_1+S_2$ is not $\mathbb{Z}^{2}$-convex.
On the other hand, if $S_1 = S_2 = S \subset \mathbb{Z}^2$ is $\mathbb{Z}^2$-convex, then $S+S$ is also $ \mathbb{Z}^2$-convex \cite[Corollary 2.54]{polytopesbook}. In $d \geq 3$ this is not true: consider the Reeve tetrahedron $S = \{(0,0,0),(1,0,0),(0,1,0),(1,1,2)\}$. Then, the set $S$ is $\mathbb{Z}^3$-convex, but $S+S$ does not contain the point $(1,1,1)$ and is thus not $\mathbb{Z}^3$-convex. 
 \end{remark}

\vspace{0.5 cm}
\noindent
{\bf Notion of isotropicity}. The isotropic constant of a function $f:\mathbb{R}^d \to \mathbb{R}_+$ is defined by  
\begin{equation} \label{Lfdef}
L_f := \left(\frac{\max_{\mathbb{R}^d}{f}}{\int_{\mathbb{R}^d}{f}}\right)^{\frac{1}{d}}\mathrm{det}(\mathrm{Cov}(f))^{\frac{1}{2d}},
\end{equation}
where $\Cov{(f)}$ is the \textit{inertia} or \textit{covariance} matrix defined, for $1 \leq i,j \leq d$, by
\begin{equation*}
    [\Cov{(f)}]_{ij}:=\frac{\int_{\mathbb{R}^{d}}x_{i}x_{j}f(x)dx}{\int_{\mathbb{R}^
    {d}}f(x)dx}-\frac{\int_{\mathbb{R}^{d}}x_{i}f(x)dx\int_{\mathbb{R}^{d}}x_{j}f(x)dx}{\Bigl(\int_{\mathbb{R}^{d}}f(x)dx\Bigr)^2} \hspace{0.1cm} .
\end{equation*}
We say that $f$ is {\em isotropic} if $\Cov{(f)} = \sigma^2\mathrm{I}_d$, for some $\sigma>0,$ where $\mathrm{I}_d$ is the $d \times d$ identity matrix. Let $K$ be a convex body in $\mathbb{R}^{d}$. Then its covariance matrix is $\Cov(K):=\Cov(\mathds{1}_K)$ and its isotropic constant is $L_K:=L_{\mathds{1}_K}$. The convex body $K$ is called \textit{isotropic} if $\mathds{1}_K$ is isotropic.

Similarly, in the discrete case $p:\mathbb{Z}^d \to \mathbb{R}_+$, we define the covariance matrix $\Cov{(p)}$
by 
\begin{equation*}
    [\Cov{(p)}]_{ij}:=\frac{\sum_{k \in \mathbb{Z}^d}k_{i}k_{j}p(k)}{\sum_k{p(k)}}-\frac{\sum_{k \in \mathbb{Z}^{d}}k_{i}p(k)\sum_{k \in \mathbb{Z}^{d}}k_{j}p(k)}{\Bigl(\sum_{k \in \mathbb{Z}^{d}}p(k)\Bigr)^2}.
\end{equation*}

\begin{definition} \label{almostisotropic}
A family $\{f_{\sigma}\}_{\sigma \in \mathbb{R}_+}$  of non-negative functions on $\mathbb{R}^d$ is {\it almost isotropic} if, as $\sigma \to \infty$,
\begin{align} \nonumber
    \Cov{(f_{\sigma})}_{i,j} &= \sigma^2 + O(\sigma) \quad \text{for } i=j, \\ \nonumber
    &= O(\sigma) \quad \text{for } i \neq j.
\end{align}

\end{definition}
\begin{remark}
It is straightforward to check that 
the family $\{f_{\sigma}\}_{\sigma \in \mathbb{R}_+}$ is {\it almost isotropic} if and only if
\begin{equation*}
||\Cov{(f_{\sigma})}-\sigma^2\mathrm{I}_d||_{\rm op} =O(\sigma),
\end{equation*}
where $||A||_{\rm op}$ is the operator norm. 
\end{remark}
\noindent
We are interested in log-concave densities $f$, for which $\det{\Cov{(f)}} \to \infty$. Thus, when we write that $f$ is almost isotropic, it is meant that $f$ represents a family of densities with the underlying parameter $\sigma := \det{\Cov{(f)}}^{\frac{1}{2d}}$ growing large.
When its not clear from the context, we write $\Cov_{\mathbb{R}}(f)$ and $\Cov{_{\mathbb{Z}}(f)}$ to distinguish between the continuous and discrete covariance matrix of $f$. 





\section{Proofs of main results} \label{relaxedsection}

Our first observation is that the class $\Cc$ defined in Definition \ref{familydef} is closed under self-convolution up to constants, as shown by the following proposition: 
\begin{proposition} \label{convpreserve}
Suppose 
$p \in \mathcal{C}(c_1(d),c_2(d),c_3(d)) $ for some constants $c_1(d),c_2(d) \geq 1, c_3(d)$ depending on the dimension only. Then the convolution of $p$ with itself $n$ times, 
$p_n := p *\stackrel{n}\ldots *p$, satisfies
$$p_n \in \mathcal{C}({n^{\frac{d}{2}}}c_1(d),n^{\frac{d}{2}}c_2^{\prime}(d,n),\sqrt{n}c_3(d))$$ for some possibly different constant $c_2^{\prime}(d,n)$ that depends on the dimension and $n$.

That is, denoting $\sigma:= \det{\Bigl(\Cov{(p)}}\Bigr)^{\frac{1}{2d}}$, we have
\begin{enumerate}
    \item \label{c1constant}
    $$\max_{k \in \mathbb{Z}^d}{p_n(k)} \leq \frac{c_1(d)}{\sigma^d}.$$ 

\item \label{c2constant} For all $k \in \mathbb{Z}^d$ with $\|k\|_{\infty} \geq \ctn \sigma$, 
$$
p_n(k) \leq \frac{\ctn}{\sigma^d}e^{-\frac{\|k\|_1}{\ctn\sigma}},
$$
where in fact we may take
$$
\ctn = \bigl(3\cdot4^{d+1}c_2(d)\bigr)^{2^n-1}.
$$
\item \label{c3constant} 
$$\sum_{i=1}^d{\sum_{k \in \mathbb{Z}^d}}{|p_n(k)-p_n(k-e_i)|}\leq \frac{c_3(d)}{\sigma}.$$
\end{enumerate}
\end{proposition}

\begin{proof}

First we note that, since $\Cov{(p_n)}=n\Cov{(p)}$, one has $\det{\Bigl(\Cov{(p_n)}}\Bigr)^{\frac{1}{2d}} = \sqrt{n}\det{\Bigl(\Cov{(p)}}\Bigr)^{\frac{1}{2d}},$ so $p_n \in \mathcal{C}(n^{\frac{d}{2}}c_1(d),n^{\frac{d}{2}}c_2^{\prime}(d,n),\sqrt{n}c_3(d))$ is indeed equivalent to the three statements. 

For \ref{c1constant} we have
$$
\max_k p_{n}(k)=\max \sum_{j \in \mathbb{Z}^d} p_{n-1}(j)p(k-j)\leq \max_{k} p(k)\sum_{j \in \Z^d}{p_{n-1}(j)}= \max_{k} p(k) \leq \frac{c_{1}(d)}{\sigma^d}. 
$$

To prove that property \ref{c2constant} is satisfied, we proceed by induction. Assume that $p,$ as well as $p_{n}$, satisfy condition \ref{A2} of Definition \ref{familydef} with the constant $\ctn$. Let $k\in \mathbb{Z}$ such that $\|k\|_{\infty} \geq \ctno \sigma.$ We have
\begin{align} \nonumber
p_{{n+1}}(k)&=\sum_{j} p(j)p_{n}(k-j) \\ \label{suma}
&=\sum_{\|j\|_{\infty}\leq c_2(d)\sigma} p(j)p_{n}(k-j)+\sum_{\substack{\|j\|_{\infty}>c_2(d)\sigma, \\ \|k-j\|_{\infty}>\ctn\sigma}}p(j)p_{n}(k-j) +\sum_{\substack{\|j\|_{\infty}>c_2(d)\sigma, \\ \|k-j\|_{\infty}\leq \ctn\sigma}}p(j)p_{n}(k-j).
\end{align}

For the first sum in \eqref{suma}, we note that for $\|j\|_{\infty}\leq c_2(d)\sigma$,
$\|k-j\|_{1}\geq \|k\|_1-\|j\|_1\geq \|k\|_{1} - \sqrt{d}\|j\|_{\infty} \geq \|k\|_{1} -\sqrt{d}c_2(d)\sigma \geq \|k\|_1 - \ctn \sigma$ and $\|k-j\|_{\infty} \geq \|k\|_{\infty} - \|j\|_{\infty} \geq \ctno \sigma - c_2(d)\sigma \geq \ctn \sigma$, by the elementary bound $\|x\|_{\infty} \leq \|x\|_{1} \leq \sqrt{d}\|x\|_{\infty}, x \in \R^{d}$. Therefore we can bound that sum as

\begin{equation} \label{Ibound}
    \sum_{\|j\|_{\infty}\leq c_2(d)\sigma}p(j)p_n(k-j)\leq \frac{\ctn}{\sigma^d}\sum_{\|j\|_{\infty}\leq c_2(d)\sigma}p(j)e^
{-\frac{\|k-j\|_1}{\ctn\sigma}} \leq \frac{\ctn}{\sigma^d}e^{-\frac{\|k\|_1}{\ctn\sigma}}.
\end{equation}

For the second sum in \eqref{suma}, using $c_2(d) \leq \ctn$, we have

\begin{align} \nonumber
&\sum_{\|j\|_{\infty}>c_{2}(d)\sigma, \|k-j\|_{\infty}>\ctn\sigma}p(j)p_n(k-j) \\ \nonumber
&\leq \frac{c_2(d)\ctn}{\sigma^{2d}}\sum_{\substack{\|j\|_{\infty}>c_2(d)\sigma, \\ \|k-j\|_{\infty}>\ctn\sigma}}{e^{\frac{-\|j\|_1-\|k-j\|_1}{\ctn\sigma}}} \\ \nonumber
&\leq \frac{c_2(d)\ctn}{\sigma^{2d}}\sum_{j \in \mathbb{Z}^d}{e^{\frac{-\|j\|_1-\mid\|k\|_1-\|j\|_1\mid}{\ctn\sigma}}}
\\ \label{IIdouble}
&\leq \frac{c_2(d)\ctn}{\sigma^{2d}}\Bigl[\sum_{0\leq m\leq \|k\|_1}{\sum_{\|j\|_1=m}{e^{\frac{-m+\|k\|_1+m}{\ctn\sigma}}}} + \sum_{ m\geq \|k\|_1}{\sum_{\|j\|_1=m}{e^{\frac{-2m+\|k\|_1}{\ctn\sigma}}}}\Bigr] .
\end{align}
The first double sum in \eqref{IIdouble} is 
\begin{align} \nonumber
   \sum_{0\leq m\leq \|k\|_1}{\sum_{\|j\|_1=m}{e^{\frac{-m+\|k\|_{1}+m}{\ctn\sigma}}}} &\leq e^{-\frac{\|k\|_1}{\ctn\sigma}}\big|\{j:\|j\|_1 \leq \|k\|_{1}\}\big| \\ \nonumber
&\leq e^{-\frac{\|k\|_1}{\ctn\sigma}}\big|\{j:\|j\|_{\infty} \leq \|k\|_{1}\}\big| \\ \label{IIfirstdoublesumlast}
&\leq e^{-\frac{\|k\|_1}{\ctn\sigma}}(2\|k\|_1+1)^d \leq 4^d \|k\|_1^d e^{-\frac{\|k\|_1}{\ctn\sigma}} \leq 4^d \sigma^{d} e^{-\frac{\|k\|_1}{2\ctn\sigma}},
\end{align}
where the last inequality holds at least as long as 
\begin{equation} \label{IIklargecondition}
\|k\|_1 \geq 2 \ctn d \sigma \Bigl(\log{(2\ctn d)}\Bigr)^2,
\end{equation}
since $x^d \leq e^{\frac{x}{2c}},$ for $x \geq 2cd\bigl(\log{(2cd)}\bigr)^2, d \geq 1,c \geq 2$.
On the other hand, the second double sum in \eqref{IIdouble} is 
\begin{align} \nonumber
\sum_{ m\geq \|k\|_1}{\sum_{\|j\|_1=m}{e^{\frac{-2m+\|k\|_1}{\ctn\sigma}}}} &\leq e^{\frac{\|k\|_1}{\ctn\sigma}}\sum_{m \geq \|k\|_1}{\big|\{j:\|j\|_{\infty}\leq m\}\bigr|e^{-\frac{2m}{\sigma \ctn}}} \\ 
\end{align}
\begin{align}
&\leq 4^de^{\frac{\|k\|_1}{\ctn\sigma}}\sum_{m \geq \|k\|_1}{m^de^{-\frac{2m}{\sigma \ctn}}} \\ \label{mdexpbound}
&\leq 4^d\sigma^de^{\frac{\|k\|_1}{\ctn\sigma}}\sum_{m \geq \|k\|_1}{e^{-\frac{3m}{2\sigma \ctn}}} \\ \label{IIseconddoublelast}
&=  4^d\sigma^de^{\frac{\|k\|_1}{\ctn\sigma}} \frac{e^{-\frac{3\|k\|_1}{2\sigma \ctn}}}{1-e^{-\frac{3\|k\|_1}{2\sigma \ctn}}} \leq 2\cdot4^d \sigma^de^{-\frac{\|k\|_1}{2\sigma \ctn}},
\end{align}
where \eqref{mdexpbound} holds at least as long as \eqref{IIklargecondition} is satisfied since $m \geq \|k\|_1$ and similarly for the last inequality.
Plugging the bounds \eqref{IIfirstdoublesumlast}, \eqref{IIseconddoublelast} in \eqref{IIdouble} we obtain 
\begin{align} \label{IIbound}
\sum_{\|j\|_{\infty}>c_{2}(d)\sigma, \|k-j\|_{\infty}>\ctn\sigma}p(j)p_n(k-j) &\leq 4^{d+1}\frac{c_2(d)\ctn^2}{\sigma^d}e^{-\frac{\|k\|_1}{2\sigma \ctn}},
\end{align}
provided that \eqref{IIklargecondition} holds. 

For the third sum in \eqref{suma}, we observe that $\|k-j\|_{\infty}\leq \ctn\sigma$ implies $\|j\|_{1}\geq \|k\|_1-\|k-j\|_1\geq \|k\|_1-\ctn\sqrt{d}\sigma$, and therefore
\begin{align} \nonumber
\sum_{\|j\|_{\infty}>c_{2}(d)\sigma, \|k-j\|_{\infty}\leq \ctn\sigma} p(j)p_n(k-j)&\leq \sum_{\|j\|_{\infty}>c_{2}(d)\sigma, \|k-j\|_{\infty}\leq c_{2}(d)\sigma} \frac{c_{2}(d)}{\sigma^d} e^{-\frac{\|j\|_1}{c_{2}(d)\sigma}}p_n(k-j) \\ \nonumber
&\leq 
\frac{\ctn}{\sigma^d}e^{\frac{-\|k\|_1}{\ctn\sigma}+\sqrt{d}}\sum_{j\in \mathbb{Z}}p(k-j) \\ \label{IIIbound}
&\leq \frac{e^{\sqrt{d}}\ctn}{\sigma^d}e^{\frac{-\|k\|_1}{\ctn\sigma}}.
\end{align}
Clearly, the worst bound out of \eqref{Ibound}, \eqref{IIbound} and \eqref{IIIbound} is \eqref{IIbound} and the dominating condition for $k$ is \eqref{IIklargecondition}.
By \eqref{suma}, we can therefore bound 
\begin{equation}
    p_{n+1}(k) \leq 3\cdot4^{d+1}\frac{c_2(d)\ctn^2}{\sigma^d}e^{-\frac{\|k\|_1}{2\sigma \ctn}} \leq \frac{\ctno}{\sigma^d}e^{-\frac{\|k\|_1}{\sigma \ctno}},
\end{equation}
with $\ctno = 3\cdot4^{d+1}c_2(d)\ctn^2 = \bigl(3\cdot4^{d+1}c_2(d)\bigr)^{2^n-1}$ by the inductive hypothesis. Note that this choice of $\ctno$ also ensures that \eqref{IIklargecondition} is satisfied and therefore the proof of \ref{c2constant}. is complete. 

For property \ref{c3constant}, let $X$ be a random vector with p.m.f. $p_X = p$. Then $S_n := \sum_{i=1}^n{X_i}$, where $X_i$ are i.i.d. copies of $X$, has p.m.f. $p_n$. We observe that
$$\sum_{i=1}^d{\sum_{k \in \mathbb{Z}^d}}{|p_n(k)-p_n(k-e_i)|}=\sum_{i=1}^{d}\|p_{S_n}-p_{S_n+e_{i}}\|_{\mathrm{TV}}.$$
For any probability measures $\mu_{1},\mu_{2},\nu_{1}$ and $\nu_{2}$,
the total variation distance satisfies the convolution inequality
\[\|\mu_{1}\ast \mu_{2}-\nu_{1}\ast \nu_{2}\|_{\mathrm{TV}}\leq \|\mu_{1}-\nu_{1}\|_{\mathrm{TV}}+\|\mu_{2}-\nu_{2}\|_{\mathrm{TV}}. \]  
This can be proved either directly or by the fact that total variation is an $f-$divergence and therefore satisfies a data processing inequality. 
Choosing $\mu_{1}=\nu_{1}=p_{S_{n-1}},$ $\mu_2 = p_X$ and $\nu_{2}=p_{X+e_{i}}$, we get
\[\|p_{S_n}-p_{S_n+e_{i}}\|_{\mathrm{TV}}\leq \|p_{X}-p_{X+e_{i}}\|_{\mathrm{TV}}  \]
Using the assumption on $X$, the last inequality yields 

\[\sum_{i=1}^{d} \|p_{S_n}-p_{S_n+e_{i}}\|_{\mathrm{TV}}\leq \sum_{i=1}^{d}\|p_{X}-p_{X+e_{i}}\|_{\mathrm{TV}}\leq \frac{c_{3}(d)}{\sigma}. \]
\end{proof}

\begin{lemma} \label{densityapproxlemma}
Let $X$ be a random vector with p.m.f. $p$ on $\mathbb{Z}^d$ satisfying conditions \ref{A2}, \ref{A3} of Definition~\ref{familydef}. 
Let $\Un =\sum_{i=1}^{n}U_{i}$ where $U_{i}$ are i.i.d. uniform random vectors on the cube $[0,1]^{d}$. 
Denote by $h_{n}^{(d)}(x)=\prod_{i=1}^{d}h_{n}(x_{i})$ the density of $U_n$ and by $f_{n}^{(d)}$ the density of $\Un+X$. Then,
\begin{enumerate}
    \item  \label{lem1-2} $\forall z\in [0,1]^{d},\forall n\geq 1$,
 $\sum_{j\in \{0,\ldots, n-1\}^{d}}h_{n}^{(d)}(j+z)=1$.  
    \item \label{improvementfofG} For $x=k+z$ where $z\in [0,1]^d$ and $k\in\Z^d$, \label{lem2-2} $|f_{n}^{(d)}(x)-p(k)|\leq \frac{2^{d-1}}{d} \frac{c_{3}}{\sigma}$ with $c_{3}$ the constant that appears in condition \ref{A3} of Definition \ref{familydef}.
    \item\label{gkexponential} For $x=k+z$ where $z\in [0,1]^d$ and $k\in\Z^d$, with $\|k\|_{\infty} \geq c_2 \sigma + n$,
    $$
         f_n^{(d)}(x) \leq \Bigl(n^d(e^{\frac{nd}{c_2\sigma}} + 1)+ 1\Bigr)\frac{c_2}{\sigma^d} e^{-\frac{\|k\|_1}{c_2\sigma}}.
    $$
\end{enumerate}

\end{lemma}
\begin{proof}
First, we prove 1.
Let us prove that the one-dimensional case holds i.e.
$\sum_{j=0}^{n-1}h_{n}(j+z)=1$, for any $z\in[0,1]$. For this, one does an induction on $n$. 
For $n=1$, since $S_{1}=U_{1}$, $h_{1}(z)=1_{[0,1]}$ so the base case is immediate. 
Let us suppose by induction that the claim holds for $n$. Noticing that the support of $h_n$ is $[0,n]$ and using the induction hypothesis, we get 
\begin{align*}
    \sum_{j=0}^{n}h_{n+1}(j+z)&=\sum_{j=0}^{n}\int_{0}^{1}h_{n}(j+z-u)du =\sum_{j=0}^{n}\left( \int_{0}^{z}h_{n}(j+z-u)du+\int_{z}^{1}h_{n}(j+z-u)du \right)\\
    &=z+\int_{0}^{z}h_{n}(n+z-u)du+\int_{z}^{1} \left(\sum_{k=0}^{n-1}h_{n}(k+1+z-u)\right)du=z+1-z=1 . 
\end{align*}
For the $d$-dimensional case, observe that since $h_{n}^{(d)}$ is a product, for all $z\in [0,1]^{d}$,
\[\sum_{j\in \{0,\ldots, n-1\}^{d}}h_{n}^{(d)}(j+z)=\sum_{j\in \{0,\ldots, n-1\}^{d}}\left(\prod_{i=1}^{d}h_{n}(z_{i}+j_{i})\right)=\prod_{i=1}^{d}\left(\sum_{j_i=0}^{n-1}h_{n}(z_{i}+j_{i})\right)=1,  \]
where, for the last equality, we have used the above one-dimensional argument.

Let us prove the second item. Using item \ref{lem1-2}, one has 
\begin{equation} \label{densitydef}
f_{n}^{d}(x)-p(k)=f_{n}^{(d)}(k+z)-p(k)=\sum_{i\in\Z^d} p(k-i)h_{n}^{(d)}(i+z)-p(k)=\sum_{i\in \mathbb{Z}^{d}}u(i)v(i),
\end{equation}
where $u(i)=p(k-i)-p(k)$, $v(i)=h_{n}^{(d)}(z+i)$. Note that the sum actually ranges over $i \in \{0,\ldots,n-1\}^d$, since $h_n^{(d)}(i+z) = 0$ outside this discrete cube. Also, for $i=(i_1,\dots,i_d)\in \mathbb{Z}^{d}$, denote
\[W(i):=\left(T(v)\right)_{i}=\sum_{j\geq i} v_{j},\]
where $j\geq i$ means that $j_{l}\geq i_{l}$ for all $1\leq l \leq d$. By a straightforward induction, 
\[v(i)=(T^{-1}(W))_{i}=\sum_{\epsilon\in \{-1,1\}^{d}}\left(\prod_{j=1}^{d}\left(-\epsilon_{j}\right)\right)W\left(i+\sum_{k=1}^{d}\frac{\epsilon_{k}+1}{2}e_{k}\right) .\]
Let us introduce the standard dot product
in $l^{2}(\mathbb{Z})$ as
$\langle z,h \rangle_{l^{2}(\mathbb{Z})}:=\sum_{i\in \mathbb{Z}^{d}}z(i)h(i) $ for any $z,h\in \mathbb{R}^{\mathbb{Z}^{d}}$. Then the sum
$\sum_{i\in \Z^{d}}u(i)v(i)$
can be rewritten as $\langle u,T^{-1}W \rangle_{l^{2}(\mathbb{Z})}$ with \[T^{-1}:=\sum_{\epsilon\in \{-1,1\}^{d}}\left(\prod_{j=1}^{d}(-\epsilon_{j})\right)S_{\epsilon} ,
\]
where $S_{\epsilon}$ is the left-shift operator, that is, $\left(S_{\epsilon}W\right)_{i}=W\left(i+\sum_{k=1}^{d}\frac{\epsilon_{k}+1}{2}e_{k}\right)$. It is well known that the adjoint $S_{\epsilon}^{\star}$ of $S_{\epsilon}$, 
is the right-shift operator or in another words $(S_{\epsilon}^{\star}(W))_{i}=W(i-\sum_{k=1}^{d}\frac{\epsilon_{k}+1}{2}e_{k})$. Thus
\begin{align*}
    \sum_{i\in \Z^{d}}u(i)v(i)&=\sum_{i\in \Z^{d}}\sum_{\epsilon\in \{-1,1\}^{d}}u_{i}\left(\prod_{j=1}^{d}\left(-\epsilon_{j}\right)\right)W\left(i+\sum_{k=1}^{d}\frac{\epsilon_{k}+1}{2}e_{k}\right)\\
    &=\langle u,T^{-1}(W)\rangle_{l^{2}(\mathbb{Z})}= \langle\left(T^{-1}\right)^{\star}u, W \rangle_{l^{2}(\mathbb{Z})}\\
    &=\sum_{i\in \Z^{d}}W(i)\sum_{\epsilon\in \{-1,1\}^{d}} \left(\prod_{j=1}^{d} (-\epsilon_{j})\right) u\left(i-\sum_{k=1}^{d}\frac{\epsilon_{k}+1}{2}e_{k}\right) 
\end{align*}
Thus, by the triangular inequality, summing up to the last coordinate which can be either $1$ or $-1$, we obtain 
\begin{align*}
\left|\sum_{i\in \mathbb{Z}^{d}} u(i)v(i)\right| &\leq \sum_{i\in \mathbb{Z}^{d}}
\sum_{\epsilon\in \{-1,1\}^{d-1}} \left|u\left(i-\sum_{k=1}^{d}\frac{\epsilon_{k}+1}{2}e_{k}\right)-u\left(i-\sum_{k=1}^{d-1}\frac{\epsilon_{k}+1}{2}e_{k}-e_{d}\right)\right|\\
&\le 2^{d-1}\sum_{i\in \mathbb{Z}^{d}}\left|u(i)-u(i-e_{d})\right| .
\end{align*}
Using the same argument for $j\in \{1,\ldots ,d\}$, instead of $j=d$ and summing over $j$ one gets 
\[\left|f_{n}(x)-p(k) \right|=\left|\sum_{i\in \mathbb{Z}^{d}}u(i)v(i) \right|\leq \frac{2^{d-1}}{d}\sum_{i\in \mathbb{Z}^{d}}\sum_{j=1}^{d}\left|p(i)-p(i+e_{j})\right|\leq
\frac{2^{d-1}}{d}\frac{c_{3}}{\sigma}, \]
where, in the last inequality, we have used that $X$ satisfies condition \ref{A3} of Definition \ref{familydef}.

Finally, for item \ref{gkexponential}, we have by \eqref{densitydef}, since $|h_n| \leq 1$,
\begin{align*}
    |f_n^{(d)}(x) - p(k)| &= \bigl|\sum_{i \in \{0,\ldots,n-1\}^d}{(p(k-i) - p(k))h_n(z+i)}\bigr| \\
    &\leq \sum_{i \in \{0,\ldots,n-1\}^d}{p(k-i)} + n^dp(k). 
\end{align*}
But, using hypothesis \ref{A2}, the assumption $\|k\|_{\infty} \geq c_2\sigma + n$ and the observations that $\|k-i\|_{\infty} \geq \|k\|_{\infty} - n \geq c_2\sigma$ and $\|k-i\|_1 \geq \|k\|_1 - nd,$ we have
\begin{align*}
    |f_n^{(d)}(x) - p(k)| &\leq n^d\frac{c_2}{\sigma^d}e^{-\frac{(\|k\|_1 - nd)}{c_2\sigma}} + n^d\frac{c_2}{\sigma^d}e^{-\frac{\|k\|_1}{c_2\sigma}}
\end{align*}
and the claimed bound follows since $f_n^{(d)}(x) \leq |f_n^{(d)}(x) - p(k)| + p(k)$. 
\end{proof}
\begin{remark}
The upper bound of Lemma \ref{densityapproxlemma}, item \ref{improvementfofG} depends on $d$, but not on $n$. Hence, for $d = 1$ our new method of proof improves the bound \cite[Eq. (21)]{gavalakis} by removing the dependence on $n$, i.e. the number of uniforms being added to the discrete random variable. 
\end{remark}

We are ready to prove Theorem \ref{diffentropyapprox}. Theorem \ref{EPI} will readily follow from this, together with the generalized EPI in $\mathbb{R}^d$.

\begin{proof}[Proof of Theorem \ref{diffentropyapprox}]

Let $F(x) = x\log{\frac{1}{x}}, x>0$ and note that $F(x)$ is increasing for $x \leq 1/e$. Denote $S_n = \sum_{i=1}^n{X_i},$  $\Un = \sum_{i=1}^n{U_i}$, and let $f_{S_n+\Un}$ be the density of $S_n+\Un$ on the $\mathbb{R}^d$ and $p_{S_n}$ the p.m.f. of $S_n$. Note that $p_{S_n}$ satisfies the three parts of Proposition \ref{convpreserve}. We have 
\begin{align} \nonumber
&h(X_1+\cdots+X_n+U_1+\cdots+U_n) \\ \label{sumbreak}
&= \sum_{k: \|k\|_{\infty}\leq \sigma^2}{\int_{k+[0,1)^d}{F(f_{S_n+\Un}(x)) dx}} + \sum_{k: \|k\|_{\infty} >  \sigma^2}{\int_{k + [0,1)^d}{F(f_{S_n+\Un}(x)) dx}}.
\end{align}
First, we bound the second term. From part \ref{gkexponential} of Lemma \ref{densityapproxlemma}, in that range of $k$ and for any $x\in k+[0,1)^d$, 
\begin{equation} \label{fsnunleqpk}
f_{S_n+\Un}(x) \leq \frac{C^{\prime}_{d,n}}{\sigma^d} 
    e^{-\frac{\|k\|_1}{\ctn \sigma}}
\end{equation}
for some constants $C^{\prime}_{d,n}, c^{\prime}(d,n)>0$ and $\sigma$ large enough. Therefore, for $\sigma$ large enough $f_{S_n+\Un} \leq 1/e$ so that $F$ is increasing. Moreover, using \eqref{fsnunleqpk} and that $\|x\|_{\infty} \leq \|x\|_{1} \leq \sqrt{d}\|x\|_{\infty}, x \in \mathbb{R}^d$, we have, for $\sigma$ large enough,
\begin{align} \label{ctsdiscretetails}
 \sum_{k: \|k\|_{\infty} > \sigma^2}{\int_{k + [0,1)^d}{F(f_{S_n+\Un}(x)) dx}}
 &\lesssim_{d,n}  \sum_{k: \|k\|_{\infty} >  \sigma^2}{\int_{k + [0,1)^d}{F\Bigl(\frac{C^{\prime}_{d,n}}{\sigma^d} 
    e^{-\frac{\|k\|_1}{\ctn \sigma}}\Bigr) dx}} \\ \nonumber
    &\lesssim_{d,n} \sum_{k: \|k\|_{\infty} > \sigma^2}{\frac{1}{\sigma^d}
    e^{-\frac{\|k\|_{\infty}}{\ctn\sigma}} \frac{\|k\|_{\infty}}{\sigma}} \\\label{tailstwoterms} &+  \sum_{k: \|k\|_{\infty} >  \sigma^2}{\frac{1}{\sigma^d} 
    e^{-\frac{\|k\|_{\infty}}{\ctn\sigma}}\log{\sigma}}.
    \end{align}
    Observing that the first term in \eqref{tailstwoterms} dominates and using a similar calculation as the one leading to \eqref{IIseconddoublelast}, we obtain 
    \begin{align} \nonumber
&\sum_{k: \|k\|_{\infty} > \sigma^2}{\int_{k + [0,1)^d}{F(f_{S_n+\Un}(x)) dx}}
  \\&\lesssim_{d,n} \frac{1}{\sigma^{d+1}} \sum_{k: \|k\|_{\infty} > \sigma^2}{
    e^{-\frac{\|k\|_{\infty}}{\ctn\sigma}} \|k\|_{\infty}} \\ \label{similarcalcmd}
    &\lesssim_{d,n} \frac{1}{\sigma^{d+1}} \sum_{m > \sigma^2}{m^{d+1}
    e^{-\frac{m}{\ctn\sigma}} }
        \\ \label{similarcalcmd2}
       &\lesssim_{d,n} \sum_{m > \sigma^2}{
    e^{-\frac{m}{2\ctn\sigma}} }  \lesssim_{d,n} \frac{e^{-\frac{\sigma}{2\ctn}}}{1-e^{-\frac{1}{2\ctn \sigma}}}
    \lesssim_{d,n} \sigma e^{-\frac{\sigma}{2\ctn}} = O(\sigma e^{-\Theta(\sigma)}).
\end{align}
We show that the first term in \eqref{sumbreak} is approximately $H(S_n)$.
To that end we apply the estimate~\eqref{cleanelemestimate} of Lemma \ref{elementaryestimatelemma} to the integrand with $\mu = \Theta_{d,n}( \frac{1}{\sigma}), D = \Theta_{d,n}(\sigma^{2d-1}), M = \sigma^{2d}, a = f_{S_n+\Un}(x)$ and $b=p_{S_n}(k).$ The assumption of the lemma is satisfied for an appropriate choice of the implied constants since by Lemma \ref{densityapproxlemma} and part \ref{c1constant} of Proposition \ref{convpreserve}, 
$$a,b \lesssim_{d,n} \max_{k\in \Z^d}{p_{S_n}(k)} + O_{d,n}\Bigl(\frac{1}{\sigma}\Bigr) \lesssim_{d,n} \frac{1}{\sigma} \simeq_{d,n} \frac{D}{M}.$$
Thus, since there are at most $\Theta(\sigma^{2d})$ elements in the set $\{k \in \mathbb{Z}^d: \|k\|_{\infty} \leq \sigma^2\}$, 
\begin{align} \nonumber
& \sum_{k: \|k\|_{\infty}\leq \sigma^2}{\Bigl|\int_{k + [0,1)^d}{F(f_{S_n+\Un}(x))dx}} - {F(p_{S_n}(k))} \Bigr| \\ \label{importantterms}
&\lesssim_{d,n} \sigma^{2d}\frac{\log{\sigma}}{\sigma^{2d+1}} + \log{\sigma}\sum_{k \in \mathbb{Z}^d}{\int_{k+ [0,1)^d}|f_{S_n+\Un}(x) - p_{S_n}(k)|dx} \\ \label{errorterm}
& \lesssim_{d,n} \frac{\log{\sigma}}{\sigma} + \log{\sigma}\sum_{k\in \mathbb{Z}^d}{\sup_{x \in k + [0,1)^d}|g_n(k, x)|} \\ \label{lastmucorr}
& \lesssim_{d,n} \frac{\log{\sigma}}{\sigma},
\end{align}
where $g_n(k,x)$ is given by Lemma \ref{densityapproxlemma} applied to the random variable $S_n$, which satisfies the assumption $p_{S_n} \in \mathcal{C}(c_1^{\prime}(d,n),c_2^{\prime}(d,n),c_3^{\prime}(d,n))$ for some constants that depend on $d,n$, and therefore 
$$\sum_k{\sup_{x \in k + [0,1)^d}|g_n(k, x)|} \lesssim \det{\Bigl(\Cov{(S_n)}\Bigr)}^{-\frac{1}{2d}} \simeq_{d,n} \frac{1}{\sigma}.$$ 
Therefore, by \eqref{lastmucorr},

\begin{align} \nonumber
\Bigl| H(S_n) - \sum_{k: \|k\|_{\infty}\leq \sigma^2}{\int_{k+ [0,1)^d}{F(f_{S_n+\Un}(x)) dx}}
 \Bigr|&\lesssim_{d,n}  \Bigl| H(S_n) -\sum_{k: \|k\|_{\infty}\leq \sigma^2}{F(p_{S_n}(k))}\Bigr|    + \frac{\log{\sigma}}{\sigma} \\  \label{lastHSnabs}
&\lesssim_{d,n} \sum_{k: \|k\|_{\infty}> \sigma^2}{F(p_{S_n}(k))} + \frac{\log{\sigma}}{\sigma} .
\end{align}
But, in view of \eqref{fsnunleqpk} and \eqref{ctsdiscretetails}, we can bound the discrete tails in the same way: 
\begin{equation} \label{discretetailsbound}
\sum_{k: \|k\|_{\infty} > \sigma^2}{F(p_{S_n}(k)) } \lesssim_{d,n} \frac{\sigma}{e^{\Theta_{d,n}(\sigma)}}.
\end{equation}
\noindent
%
%
Plugging the bound \eqref{discretetailsbound} into \eqref{lastHSnabs}, we get 
\begin{align} \nonumber
&\Bigl| H(S_n) - \sum_{k: \|k\|_{\infty}\leq \sigma^2}{\int_{k+ [0,1)^d}{F(f_{S_n+\Un}(x)) dx}} \Bigr| \\ \nonumber
&\leq  O_{d,n}\Bigl(\frac{\sigma}{e^{\Theta_{d,n}(\sigma)}}\Bigr) + O_{d,n}\Bigl(\frac{\log{\sigma}}{\sigma}\Bigr)  \\ \nonumber
&= O_{d,n}\Bigl(\frac{\log{\sigma}}{\sigma}\Bigr).
\end{align}
Finally, by \eqref{sumbreak} and the exponential bounds on the entropy tails \eqref{lastmucorr}, \eqref{discretetailsbound}, we obtain 
\begin{align}
    \bigl|h(S_n+U^{(n)}) - H(S_n)\bigr| &= O_{d,n}\Bigl(\frac{\log{\sigma}}{\sigma}\Bigr) + O_{d,n}\Bigl(\frac{\sigma}{e^{\Theta_{d,n}(\sigma)}}\Bigr) \\ \nonumber
    &= O_{d,n}\Bigl(\frac{\log{\sigma}}{\sigma}\Bigr),
\end{align}
completing the proof. 

\end{proof}

\begin{proof}[Proof of Theorem \ref{EPI}]
By the continuous EPI \eqref{cntsepid},
\begin{equation} \nonumber
    h\Bigl(\sum_{i=1}^{n+1}{X_i+U_i}\Bigr) \geq h\Bigl(\sum_{i=1}^{n}{X_i+U_i}\Bigr) + \frac{d}{2}\log{\Bigl(\frac{n+1}{n}\Bigr)}
\end{equation}
and by Theorem \ref{diffentropyapprox} applied to the differential entropies on both sides
\begin{equation} \nonumber
H\Bigl(\sum_{i=1}^{n+1}{X_i}\Bigr) + O_{d,n}\Bigl(\frac{\log{\sigma}}{\sigma}\Bigr) \geq H\Bigl(\sum_{i=1}^{n}{X_i}\Bigr) + O_{d,n}\Bigl(\frac{\log{\sigma}}{\sigma}\Bigr) +  \frac{d}{2}\log{\Bigl(\frac{n+1}{n}\Bigr)}.
\end{equation}
But by \eqref{gaussmax}, $O_{d,n}\Bigl(\frac{\log{\sigma}}{\sigma}\Bigr) = O_{d,n}\Bigl({H(X_1)}{e^{-\frac{1}{d}H(X_1)}}\Bigr)$
and the result follows. 
\end{proof}

\begin{remark}
The rate can be improved to superpolynomial by splitting the sum according to $\sigma^{1+\epsilon}$ and optimising over $\epsilon$.
\end{remark}

\begin{lemma}[\cite{gavalakis}] \label{elementaryestimatelemma}
Let $D, M \geq 1$ and, for $x > 0,$ consider $G(x) = F(x) -x\log{M}$, where $F(x) = -x\log{x}.$
Then, for $0 \leq a,b \leq \frac{D}{M}$ and any $0 < \mu < \frac{1}{e},$
we have the estimate
\begin{equation} \nonumber
|G(b) - G(a)| \leq  \frac{2\mu}{M}\log{\frac{1}{\mu}} + |b-a|\log{\frac{eD}{\mu}},
\end{equation}
whence 
\begin{equation} \label{cleanelemestimate}
|F(b) - F(a)| \leq \frac{2\mu}{M}\log{\frac{1}{\mu}} + |b-a|\log{\frac{eDM}{\mu}}.
\end{equation}

\end{lemma}
\begin{proof}
See \cite{gavalakis}.
\end{proof}

\section{Discrete isotropic log-concave random vectors} \label{logconcavesection}

This section is devoted to the proof (see end of Section \ref{cntsdiscretesection}) of the following result, which shows that discrete log-concave extensible random vectors with isotropic extension satisfy the properties of Definition \ref{familydef}. In Section \ref{pisotropicsection} we show that this can be generalized to discrete log-concave extensible random vectors with almost isotropic extension. 
\begin{theorem} \label{logconcavesatisfyTh}
Suppose $p$ is a centered log-concave p.m.f. on $\mathbb{Z}^d$ with isotropic extension and covariance matrix $\Cov{(p)} = \sigma^2\mathrm{I}_{d}$. Then there are constants $c_1(d),c_2(d),c_3(d)$ such that $p \in \mathcal{C}\left(c_1(d),c_2(d),c_3(d)\right)$, provided that $\sigma = \det\left(\Cov_{\mathbb{Z}^d}(p)\right)^{\frac{1}{2d}}$ is large enough depending on $d$, where $\mathcal{C}\left(c_1(d),c_2(d),c_3(d)\right)$ is defined in Definition \ref{familydef}. That is, if $\sigma$ is large enough depending on $d$
\begin{enumerate}
\item \label{logconcave1} $$\max_{k\in \Z^d}p(k) \leq \frac{c_1(d)}{\sigma^d}.
$$
\item \label{logconcave2} For every $k$ with $\|k\|_{\infty} \geq c_2(d)\sigma$
$$
p(k) \leq \frac{C_2(d)}{\sigma^d}e^{-\frac{\|k\|_1}{c_2(d)\sigma}}.
$$
\item \label{logconcave3} $$
\sum_{i=1}^d\sum_{k \in \mathbb{Z}^d}|p(k) - p(k-e_i)| \leq \frac{c_3(d)}{\sigma}.
$$
\end{enumerate}

\end{theorem}

\subsection{Ball's bodies} \label{ballbodiessec}
In what follows, we set $\omega_{d}$ to be the volume of the unit Euclidean ball denoted by $B_{2}^{d}$. For a given convex body $K\in \mathbb{R}^{d}$ with $0$ in its interior, its \textit{inradius} $r(K)$ is defined as the largest $r>0$ for which $rB_{2}^{d}\subset K$ and its \textit{circumradius} $R(K)$ is defined as the smallest $R>0$ such that $K\subset RB_{2}^{d}$. The support function of $K$ is defined as $h_{K}(x)=\max_{y\in K}\langle x,y \rangle$ and the radial function of $K$ is $\rho_K(x)=\inf\{t>0;\ tx\in K\}=1/\|x\|_K$, where $\|x\|_K$ denotes the gauge associated to $K$.

For most results in this section we will assume that $f:\mathbb{R}^{d}\rightarrow \mathbb{R}$ is a centered, isotropic, log-concave density, that is 
$
\int_{\mathbb{R}^{d}} f=1$,
$
\int_{\mathbb{R}^{d}} xf=0$ and
$\int_{\mathbb{R}^{d}}x^Txf=\sigma^{2}\mathrm{I}_d.
$
To this function $f$, we attach its Ball's bodies.
This important family of bodies was introduced by Ball \cite{Ball88}. We refer to the book \cite{Giannopoulos14} for the properties of these.
\begin{definition}
Let $f:\mathbb{R}^{d}\rightarrow [0,\infty)$ be an integrable, centered log-concave function. 
For any $p>0$, the set $K_{p}(f)$ is defined as follows
\begin{equation*}
    K_{p}(f):=\left\{x\in \mathbb{R}^{d}: \int_{0}^{\infty} pr^{p-1}f(rx)dr\geq f(0)\right\} .
\end{equation*}
Ball \cite{Ball88} established that the set $K_p(f)$ is a convex body. Moreover, its radial function is 
\begin{equation*}
    \rho_{K_{p}(f)}(x)=\Bigl(\frac{1}{f(0)}\int_{0}^{\infty}pr^{p-1}f(rx)dr\Bigr)^{\frac{1}{p}} \hspace{0.2cm} \text{for} \hspace{0.2cm} x\neq 0.
\end{equation*}
From integration in polar coordinates (see \cite[Proposition 2.5.3]{Giannopoulos14}), it follows that $K_{d+1}(f)$ is centered and, for any $p\ge0$ and $u\in \mathbb{S}^{d-1}$,
\begin{equation}\label{int-K-p-f}
\int_{K_{d+p}(f)}|\langle x,u\rangle|^pdx=\frac{1}{f(0)}\int_{\R^d}|\langle x,u\rangle|^pf(x)dx.
\end{equation}
Finally, the following inclusion relations between $K_{p}(f)$ and $K_{q}(f)$ hold \cite[Proposition 2.5.7]{Giannopoulos14}, for $p<q$,
\[
\frac{\Gamma(p+1)^{\frac{1}{p}}}{\Gamma(q+1)^{\frac{1}{q}}}K_{q}(f)\subset K_{p}(f)\subset e^{\frac{d}{p}-\frac{d}{q}} K_{q}(f).
\]
Applying these inclusions first for $p=d$ and $q=d+1$ then for $p=d+1$ and $q=d+2$ and using classical inequalities on Gamma functions, it follows that there exist universal constants  $0<c_1<c_2$ such that for any dimension $d\ge1$, one has 
\begin{equation}\label{inclusion-d-d+1}
c_1 K_{d+1}(f)\subset K_{d}(f)\subset c_2 K_{d+1}(f) \quad \hbox{and}\quad c_1 K_{d+2}(f)\subset K_{d+1}(f)\subset c_2 K_{d+2}(f).
\end{equation}
\end{definition}

We shall apply to $K_{d+1}(f)$ the following theorem due to Kannan, Lov{\'a}sz and Simonovits \cite[Theorem 4.1]{KLS} to which we give a new and simple proof. 

\begin{theorem}\label{KLS-h_K}
Let $K$ be a centered convex body in $\R^d$ and $u\in \mathbb{S}^{d-1}$. Then
\begin{equation}\label{h_K-x-u-square}
   \frac{h_K(u)^2}{d(d+2)}\le\frac{1}{|K|}\int_K\langle x,u\rangle^2dx=\langle\Cov(K)u,u\rangle\le\frac{d}{d+2}h_K(u)^2.
\end{equation}
\end{theorem}

\begin{proof}
Let $f(t)=|\{x\in K; \langle x,u\rangle=t\}$. Denote by $[-b,a]$ its support. Then one has $a=h_K(u)$, 
\[
\int_\R f(t)dt=|K|, \quad \int_\R tf(t)dt=\int_K\langle x,u\rangle dx\quad \hbox{and}\quad \int_\R t^2f(t)dt=\int_K\langle x,u\rangle^2 dx.
\]
We define $g:\R\to\R_+$ by 
\[
g(t)=\alpha\left(1+\frac{t}{da}\right)^{d-1}\mathds{1}_{[-da,a]}(t),\quad\hbox{where}\ \alpha=\frac{1}{a\left(1+\frac{1}{d}\right)^d}\int_\R f(t)dt.
\]
And let $h:\R\to\R_+$ be defined by $h(t)=dh(-dt)$. Then, it is not difficult to see that 
\begin{equation}\label{f-g-h}
\int_\R g(t)dt= \int_\R h(t)dt=\int_\R f(t)dt \quad \hbox{and}\quad 
\int_\R tg(t)dt= \int_\R th(t)dt=\int_\R tf(t)dt=0.
\end{equation}
We shall prove that 
\[
\int_\R t^2g(t)dt\le \int_\R t^2f(t)dt \le \int_\R t^2h(t)dt,
\]
from which the inequality \eqref{h_K-x-u-square} follows immediately by a simple calculation.
If $f\neq g$, it follows from \eqref{f-g-h} that there exist at least two points $t_1(g)<t_2(g)$ such that $g-f$ changes sign at these points and, in the same way, if $f\neq h$ there exists at least two points $t_1(h)<t_2(h)$ such that $f-h$ changes sign at these points.
Since the function $f^{1/(d-1)}$ is concave on its support $[-b,a]$ and the functions $g^{1/(d-1)}$ and $h^{1/(d-1)}$ are affine on their respective supports $[-da,a]$ and $[-\frac{a}{d},a]$, it follows that the functions $g-f$ and $f-h$ have exactly two sign changes and that they are non positive on $[t_1,t_2]$ and non negative outside $[t_1,t_2]$. Let us define $v$ to be either $g-f$ or $f-h$. Then, for any $t\in\R$, one has $(t-t_1)(t-t_2)v(t)\ge0$.
Integrating this inequality, we deduce that $\int t^2v(t)dt\ge0$, which is the result.
\end{proof}



We shall deduce from the preceding theorem the following important technical lemma. We note that Lemma \ref{KB} (and its consequences) is the only part, where the isotropicity assumption is used. 

\begin{lemma}\label{KB}
Let $d\ge1$ be an integer. There exist two constants $0<C_d'<C_d$ such that for any $f:\mathbb{R}^{d}\rightarrow \mathbb{R_+}$  centered, isotropic, log-concave density and for every $\theta\in \mathbb{S}^{d-1}$,
\begin{equation}\label{r-d-1-f-bounds}
   C_{d}^{\prime}\leq \left(\int_{0}^{\infty} dr^{d-1}f(r\theta)dr\right)^\frac{1}{d}\leq C_{d} ,
\end{equation}
where $C_{d}$ and $C_{d}^{\prime}$ are constants depending only on the dimension $d$. In fact, we may take 
\begin{equation}\label{constant1}
 C_{d}^{\prime}=\frac{c_1^{d+2}}{\sqrt{2\pi} e^\frac{3}{2}}    \quad\hbox{and}\quad   C_{d}=(d+1)c_2^{d+2}L_d,
\end{equation}
 where $c_1$ and $c_2$ are the absolute constants appearing in \eqref{inclusion-d-d+1}.
\end{lemma}

\begin{proof}
The function $f$ being isotropic, we have $\Cov(f)=\sigma^2I_d$, for some $\sigma>0$ and $\int f=1$, thus $L_f=\max(f)^\frac{1}{d}\sigma$.
Note that proving \eqref{r-d-1-f-bounds} is equivalent to proving that 
\[
C_d'B_2^d\subset f(0)^\frac{1}{d}K_{d}(f)\subset C_dB_{2}^{d}.
\]
Applying Theorem \ref{KLS-h_K} to the centered body $K_{d+1}(f)$, we get that for any $u\in \mathbb{S}^{d-1}$,
\begin{equation}\label{h_K-d-x-u-square}
   \frac{h_{K_{d+1}(f)}(u)^2}{d(d+2)}\le\frac{1}{|K_{d+1}(f)|}\int_{K_{d+1}(f)}\langle x,u\rangle^2dx\le\frac{d}{d+2}h_{K_{d+1}(f)}(u)^2.
\end{equation}
Then, we use \eqref{int-K-p-f} for $p=2$ to get that 
\[
\int_{K_{d+2}(f)}\langle x,u\rangle^2dx=\frac{1}{f(0)}\int_{\R^d}\langle x,u\rangle^2f(x)dx=\frac{\sigma^2}{f(0)}.
\]
This gives 
\[
c_1^{d+2}\frac{\sigma^2}{f(0)}\le\int_{K_{d+1}(f)}\langle x,u\rangle^2dx\le c_2^{d+2}\frac{\sigma^2}{f(0)}.
\]
Using again the inclusion relations \eqref{inclusion-d-d+1}, we get upper and lower bound of $|K_{d+1}(f)|$  by $|K_{d}(f)|=1/f(0)$, by \eqref{int-K-p-f} for $p=0$. Thus, altogether we have 
\[
c_1^{2d+2}\sigma^2\le\frac{1}{|K_{d+1}(f)|}\int_{K_{d+1}(f)}\langle x,u\rangle^2dx\le c_2^{2d+2}\sigma^2.
\]
Using these inequalities, \eqref{h_K-d-x-u-square} and the inclusion relations \eqref{inclusion-d-d+1}, we get 
\[
\sqrt{\frac{d+2}{d}}c_1^{d+2} \sigma B_2^d\subset  K_d(f)\subset\sqrt{d(d+2)}c_2^{d+2} \sigma B_2^d.
\]
From \cite[Theorem 4]{fradelizi}, we have $f(0)\ge e^{-d}\max(f)$ hence
\[
\frac{L_f}{e}\le f(0)^\frac{1}{d}\sigma=\left(\frac{f(0)}{\max(f)}\right)^\frac{1}{d}L_f\le L_d.
\]
Using that 
$L_f\ge  1/\sqrt{2\pi e}$, we conclude. 
\end{proof}

\subsection{A concentration lemma} \label{concentrationlemmasec}
The following concentration lemma will be required to establish the concentration property 2.
\begin{lemma}[Concentration Lemma]\label{LC}
Let $c_{d}:=3^{\frac{1}{d}}C_d$\label{constanti}, where $C_d$ is the constant from \eqref{constant1}. Then, for every log-concave, isotropic, centered density function $f$ and for every $x \in \mathbb{R}^d$ such that $\|x\|_2\geq c_{d}/f(0)^{\frac{1}{d}}$,  
\begin{equation} \label{tailboundr0}
    f(x) \leq f(0) 
    2^{-\|x\|_2\frac{f(0)^{\frac{1}{d}}}{c_d}}.
\end{equation}
\end{lemma}

\begin{proof}
Let $x_{\max}\in \mathbb{R}^{d}$ be a value where the maximum of $f$ is attained and for every $\theta\in \mathbb{S}^{d-1}$, let  $r_{\max}(\theta)\in \mathbb{R}^{d}$ be a value where the maximum of $r\mapsto f(r\theta)$ is attained. Since $f$ is log-concave, for every $\theta \in \mathbb{S}^{d-1}$, the function $r\mapsto f(r \theta)$ is non-increasing for $r \geq r_{\max}(\theta)$. Therefore, for every $r_0(\theta)\ge r_{\max}(\theta)$,
\begin{equation} \label{r0first}
\int_0^{\infty}dr^{d-1}f(r\theta)dr \geq \int_{r_{\max}(\theta)}^{r_{0}(\theta)}dr^{d-1}f(r\theta)dr \geq f(r_0(\theta))(r_0(\theta)^d-r_{\max}(\theta)^{d}).
\end{equation}
We choose $r_{0} = r_{0}(\theta)$ as 
\begin{equation} \label{r0estimate}
r_0(\theta)= \Bigl(2d\int_0^{\infty}{\frac{r^{d-1}f(r\theta)}{f(0)}dr}+r_{\max}^{d}\Bigr)^{\frac{1}{d}} \hspace{0.1cm},
\end{equation}
so that, 
by \eqref{r0first}, one has
$f(r_0(\theta)) \leq \frac{f(0)}{2}$.
Since $r\mapsto f(r \theta)$ is non-decreasing on $[0,r_{\max}(\theta)]$, one has
\begin{align*}
    \int_{0}^{+\infty} dr^{d-1}f(r\theta)dr\geq \int_{0}^{r_{\max}(\theta)}dr^{d-1}f(0)dr 
    =f(0)r_{\max}(\theta)^{d}, 
\end{align*}
and therefore, by Lemma \ref{KB},
\begin{equation*}
    r_{\max}(\theta)\leq \frac{\big(d\int_{0}^{+\infty}r^{d-1}f(r\theta)dr\big)^{\frac{1}{d}}}{f(0)^{\frac{1}{d}}} 
    \leq \frac{C_{d}}{f(0)^{\frac{1}{d}}}, 
\end{equation*}
where $C_{d}$ is given by \eqref{constant1}.
Thus, by \eqref{r0estimate} and the definition of $c_d$,
\begin{equation*}
    r_{0}(\theta)\leq \left(\frac{3}{f(0)}\right)^{\frac{1}{d}}C_d=\frac{c_d}{f(0)^{\frac{1}{d}}}.
\end{equation*}
Since for every $r \geq r_0(\theta)$, one has $r_{0}(\theta)=\frac{r_{0}(\theta)}{r}\cdot r+(1-\frac{r_{0}(\theta)}{r})\cdot 0$, by
log-concavity, we deduce
\begin{align} \nonumber
f(r) \leq f(0)\left(\frac{f(r_0(\theta))}{f(0)}\right)^{\frac{r}{r_0(\theta)}}
\leq f(0)2^{-{\frac{r}{r_0(\theta)}}}\le f(0)2^{-{\frac{rf(0)^{\frac{1}{d}}}{c_d}}} .
\end{align}
Hence, if $\|x\|_2\geq c_{d}/f(0)^{\frac{1}{d}}$, then 
\begin{equation*}
    f(x) \leq f(0) 2^{-\|x\|_2\frac{f(0)^{\frac{1}{d}}}{c_d}}.
\end{equation*}

\end{proof}

\subsection{Sum of maxima of isotropic log-concave functions} \label{sumofmaximasec} 

The following lemma bounds the sum of the maxima of a log-concave density using our previous concentration result.

\begin{lemma} \label{sumofmaximalemma}
Let $f$ be a centered, isotropic, log-concave density on $\mathbb{R}^d$ with covariance $\sigma^2 \mathrm{I}_{\mathrm{d}}$.  Let $0\le i\le 2$ and $0\le j\le d-1$. Then, as $\sigma \to \infty$,
\begin{equation} \label{sumofmaxima}
    \sum_{l\in \mathbb{Z}^{d-j}}|l_1|^i\max_{k\in \mathbb{Z}^j}f(k_1,\dots,k_j,l_1,\dots,l_{d-j})={O_d}\Bigl(\frac{1}{\sigma^{j-i}}\Bigr).
\end{equation}
\end{lemma}
\begin{proof}
Set $\lambda=c_df(0)^{-1/d}$, with $c_d$ is given in Lemma \ref{LC}. Then $f(x)\le f(0)2^{-\|x\|_2/\lambda}$, for $\|x\|_2\ge\lambda$. Moreover, from the definition of $L_f$, its bounds and the inequality $f(0)\ge e^{-d}\max(f)$, we have 
\[
\frac{c_d\sigma}{L_d}\le \lambda=c_df(0)^{-1/d}\le \frac{c_de\sigma}{L_f}\le c_d\sqrt{2\pi e^3}\sigma.
\]
Note that 
\begin{equation*}
\sum_{l\in \mathbb{Z}^{d-j}}|l_1|^i\max_{k\in \mathbb{Z}^j}f(k,l)=\sum_{l\in \mathbb{Z}^{d-j},\|l\|_{\infty} \leq \lambda}|l_1|^i\max_{k\in \mathbb{Z}^j}f(k,l)+\sum_{l\in \mathbb{Z}^{d-j},\|l\|_{\infty}>\lambda}|l_1|^i \max_{k\in \mathbb{Z}^j}f(k,l) ,
\end{equation*}
The first sum is easily upper bounded by ${O_d}(\frac{1}{\sigma^{j-i}})$. Indeed, 
\[
\sum_{l\in \mathbb{Z}^{d-j},\|l\|_{\infty} \le\lambda}|l_1|^i\max_{k\in \mathbb{Z}^j}f(k,l)\leq (\max f) (2\lambda)^{d+i-j}=\frac{L_f^d}{\sigma^d}(2\lambda)^{d+i-j}\le \frac{L_d^d(2c_d\sqrt{2\pi e^3})^{d+i-j}}{\sigma^{j-i}} ={O_d}\left(\frac{1}{\sigma^{j-i}}\right).
\]
Using the tails estimates of Lemma \ref{LC} and the fact that for $l\in\mathbb{Z}^{d-j}$, one has $\|l\|_2\ge\|l\|_\infty$ and $\|l\|_2\ge\frac{\|l\|_1}{\sqrt{d-j}}\ge\frac{\|l\|_1}{\sqrt{d}} $ the second sum can be expressed as follows
\begin{align*}
\sum_{l\in \mathbb{Z}^{d-j},\|l\|_{\infty}>\lambda}|l_1|^i\max_{k\in \mathbb{Z}^j}f(k,l)
&\leq f(0) \sum_{l\in \mathbb{Z}^{d-j},\|l\|_{\infty}>\lambda}|l_1|^i2^{-\frac{\|l\|_1}{\lambda\sqrt{d}}}\\
&= f(0)\left(\sum_{|l_1|>\lambda}|l_1|^{i}2^{-\frac{|l_1|}{\lambda\sqrt{d}}}\right)\left(\sum_{|n|>\lambda}2^{-\frac{|n|}{\lambda\sqrt{d}}}\right)^{d-j-1}.
\end{align*}
Putting $x=2^{-1/\lambda\sqrt{d}}$ and assuming that $\lambda\sqrt{d}\ge1$, we have $x\in[\frac{1}{2},1)$ and $\frac{1}{2}\le\lambda\sqrt{d}(1-x)\le\ln(2)$. Thus 
\[
\sum_{|n|>\lambda}2^{-\frac{|n|}{\lambda\sqrt{d}}}=2\sum_{n>\lambda}x^{n}\le\frac{2}{1-x}\le 4\lambda\sqrt{d}. 
\]
Hence, for $i=0$, replacing $\lambda$ by its value, for $\sigma\ge \frac{L_d}{\sqrt{d}c_d}$, we have $\lambda\ge 1/\sqrt{d}$ and thus
\[
\sum_{l\in \mathbb{Z}^{d-j},\|l\|_{\infty}>\lambda}\max_{k\in \mathbb{Z}^j}f(k,l)\le f(0)\left(4\lambda\sqrt{d}\right)^{d-j}=\left(4\sqrt{d}c_d\right)^{d-j}f(0)^\frac{j}{d}
\le \frac{(4\sqrt{d}c_d)^{d-j}L_d^j}{\sigma^j}={O_d}\left(\frac{1}{\sigma^j}\right).
\]
And, for $i=1$, it is not difficult to see that, for $\lambda\ge 1/\sqrt{d}$, we have
\[
\sum_{|l_1|>\lambda}|l_1|2^{-\frac{|l_1|}{\lambda\sqrt{d}}}=2\sum_{n>\lambda}nx^n=2x\left(\sum_{n>\lambda}x^n\right)'\le \frac{6}{(1-x)^2}\le 24d\lambda^2.
\]
Thus, again for $\sigma\ge \frac{L_d}{\sqrt{d}c_d}$, which ensures that $\lambda\ge 1/\sqrt{d}$, we deduce
\[
\sum_{l\in \mathbb{Z}^{d-j},\|l\|_{\infty}>\lambda}|l_1|\max_{k\in \mathbb{Z}^j}f(k,l)\le f(0)\times 24d\lambda^2\left(4\lambda\sqrt{d}\right)^{d-j-1}\le 2f(0)\left(4\lambda\sqrt{d}\right)^{d-j+1}={O_d}\left(\frac{1}{\sigma^{j-1}}\right).
\]
In the same way, for $i=2$, we easily get that, for $\lambda\ge 1/\sqrt{d}$,
\[
\sum_{|l_1|>\lambda}l_1^22^{-\frac{|l_1|}{\lambda\sqrt{d}}}=2\sum_{n>\lambda}n^2x^n=2x^2\left(\sum_{n>\lambda}x^n\right)''+2\sum_{n>\lambda}nx^n\le \frac{14}{(1-x)^3}+\frac{12}{(1-x)^2}\le 26(2\lambda\sqrt{d})^3.
\]
Thus, again for $\sigma\ge \frac{L_d}{\sqrt{d}c_d}$, which ensures that $\lambda\ge 1/\sqrt{d}$, we deduce
\[
\sum_{l\in \mathbb{Z}^{d-j},\|l\|_{\infty}>\lambda}l_1^2\max_{k\in \mathbb{Z}^j}f(k,l)\le f(0)\times 26(2\lambda\sqrt{d})^3\left(4\lambda\sqrt{d}\right)^{d-j-1}\le 4f(0)\left(4\lambda\sqrt{d}\right)^{d-j+2}={O_d}\left(\frac{1}{\sigma^{j-2}}\right).
\]

\end{proof}

\subsection{Approximating the integral, mean and covariance discretely} \label{cntsdiscretesection}

In this section, we will approximate the isotropic constant of an isotropic, log-concave density $f\in \mathbb{R}^{d}$ in a discrete way. To do so, we will first approximate the integral as a sum and subsequently approximate the continuous covariance of $f$ by its discrete covariance. 


\begin{proposition}\label{sum-int}
Let $f: \mathbb{R}^{d} \rightarrow \mathbb{R}$ be a  log-concave isotropic density function with covariance matrix of the form $\sigma^2 \mathrm{I}_{\mathrm{d}}$. Then 
\begin{equation*}
    \left|\int_{\mathbb{R}^{d}} fdx-\sum_{k\in \mathbb{Z}^{d}}f(k)\right|={O_d}\Bigl(\frac{1}{\sigma}\Bigr) \hspace{0.1cm} .
\end{equation*}
\end{proposition}
\begin{proof}

Let us proceed by induction. Let us first consider the one dimensional base case. Let $f:\mathbb{R}\rightarrow \mathbb{R}$ be a log-concave function and let suppose that the maximum of $f$ is attained at $x_{0}\in \mathbb{R}$. Let $k_{0}\in \mathbb{Z}$ such that $k_{0}\leq x_{0}< k_{0}+1$. Then, one has 

\begin{align*}
    \int_{\mathbb{R}} fdx &=\sum_{k\in \mathbb{Z}} \int_{k}^{k+1}f(x)dx \\
    &\geq \sum_{k<k_{0}}f(k)+\sum_{k\geq k_{0}+1}f(k+1)+\min\{f(k_{0}),f(k_{0}+1)\}\\ 
    &=\sum_{k\in \mathbb{Z}} f(k)-f(k_{0})-f(k_{0}+1)+\min\{f(k_{0}),f(k_{0}+1)\} \\
    &=\sum_{k\in \mathbb{Z}} f(k)-\max\{f(k_{0}),f(k_{0}+1)\} \\
    &= \sum_{k\in \mathbb{Z}} f(k)-\max_{\mathbb{Z}} f.
\end{align*}
The reverse inequality is  analogous and thus one gets 
\begin{equation}\label{int-sum-dim-1}
    \left|\int_{\mathbb{R}}fdx-\sum_{k\in \mathbb{Z}}f(k)\right|\leq \max_{\mathbb{Z}} f\leq \max_{\mathbb{R}} f.
\end{equation}
For $d=1$, we get $\max_{\mathbb{R}}(f)=\frac{L_f}{\sigma}\leq \frac{1}{\sigma}$, by \cite{fradelizi2}, which proves the base case. Notice that the inequality \eqref{int-sum-dim-1} holds as soon as $f$ is quasi-concave on $\R$, \rm{i.e.} if $f$ is first non-decreasing then non-increasing and not necessarily a density.

Let $x\in \mathbb{R}$, $y\in \mathbb{R}^{d-1}$ and let 
$F(y):=\int_{\mathbb{R}}f(x,y)dx$.
Note that $F$ is log-concave as the marginal of a log-concave function. Moreover $F$ is an isotropic density and $\Cov_{\R^{d-1}}(F)=\sigma^2 I_{d-1}$. One has
\[
 \left|\int_{\mathbb{R}^{d}} f(x,y)dxdy-\sum_{k\in \mathbb{Z}^{d}}f(k)\right|\le \left|\int_{\mathbb{R}^{d-1}} F(y)dy-\sum_{y\in \mathbb{Z}^{d-1}}F(y)\right|+\sum_{y\in \mathbb{Z}^{d-1}}\left|\int_\R f(x,y)dx-\sum_{k\in\mathbb{Z}}f(k,y)\right|.
\]
By the inductive hypothesis the first term is upper bounded by $C_{d-1}/\sigma$ for some constant $C_{d-1}>0$. Since for every $y\in\mathbb{Z}^{d-1}$, the function $x\mapsto f(x,y)$ is quasi-concave, one may use \eqref{int-sum-dim-1} to get that the second term is upper bounded by 
\[
\sum_{y\in \mathbb{Z}^{d-1}}\max_{x\in \mathbb{Z}} f(x,y).
\]
We conclude by Lemma \ref{sumofmaximalemma}.
\end{proof}

{\bf A remark from convex geometry: the case of isotropic convex bodies.} Applying the previous proposition to $f=\mathds{1}_{K}/|K|_d$, where $K$ is an isotropic convex body in $\R^d$ and using that, in this case, $\sigma=L_K|K|^{1/d}$,  we deduce the following proposition.
\begin{proposition} For any isotropic convex body $K$ in $\mathbb{R}^{d}$, 
\begin{equation*}
   \left|\frac{ \#(K\cap \mathbb{Z}^{d}) }{|K|_{d}}-1\right|\le \frac{M_d}{|K|^\frac{1}{d}} \hspace{0.1cm},
\end{equation*}
for some constant $M_d>0$, depending only on the dimension,
where $\#K$ denotes the cardinality of a discrete set $K$. 
\end{proposition}
\begin{remark}\label{counterexample}
The hypothesis of isotropicity is necessary, since in $\mathbb{Z}^{d}$, for $d\ge2$,
it is easy to construct convex bodies that do not contain integer points, but whose volumes are increasingly large. 
\end{remark}

\begin{proposition}\label{covdis}
Let $f:\mathbb{R}^{d}\rightarrow \mathbb{R}$ be a centered, isotropic, log-concave density. Then, for $1\le i\le d$,
\begin{equation*}
   \sum_{k\in \mathbb{Z}^{d}}k_{i}f(k)= \int_{\mathbb{R}^{d}}x_{i}f(x)dx +O_d(1)=O_d(1). \hspace{0.1cm} 
\end{equation*}
In particular, for $d=1$, the following finite bound holds for any log-concave centered  integrable function
\begin{equation*}
    \left|\int_{\mathbb{R}}xf(x)dx-\sum_{k\in \mathbb{Z}}kf(k)\right|\leq (e+1)\sum_{k\in \mathbb{Z}} f(k) \hspace{0.1cm} .
\end{equation*} 
 \end{proposition}


\begin{proof}

The proof is by induction again. Let us make a one dimensional remark that will be used on the base case of the induction. Let $f:\mathbb{R}\rightarrow \mathbb{R}$ be a centered log-concave function and let suppose that the maximum of $f$ is attained at $x_{0}\in \mathbb{R}$ and let us suppose without loss of generality that $x_{0}>0$. Then since $f$ is log-concave and by the  inequality \cite[Theorem 4]{fradelizi} one has
\begin{equation}\label{maxlog}
    \int_{\mathbb{R}} f(x)dx\geq \int_{0}^{x_{0}}f(x)dx\geq f(0)x_{0}\geq \frac{\max f}{e} x_{0} \hspace{0.1cm} .
\end{equation}
Hence, we get $x_{0}\max f \leq e \int_{\mathbb{R}}f(x)dx$. The case of the summation is analogous. Let $k_{0}\in \mathbb{Z}$ be such that $k_{0}\leq x_{0}< k_{0}+1$ . Then, we deduce
\begin{equation}\label{conseqfrad}
    \sum_{k\in \mathbb{Z}} f(k)\geq \sum_{k=0}^{k_{0}} f(k)\geq (k_{0}+1)f(0)\geq \frac{\max f}{e}(k_{0}+1) \hspace{0.1cm} .
\end{equation}
Let us start the induction. Let $f$ be a centered, isotropic, log-concave density and let suppose that the maximum of $f$ is attained at $x_{0}$. We decompose the integral $\int_{\mathbb{R}} xf(x)dx$ as follows
\begin{align*}
\int_{\mathbb{R}} xf(x)dx &= \sum_{k<0} \int_{k}^{k+1}xf(x)dx + \sum_{k=0}^{k_{0}-1} \int_{k}^{k+1}xf(x)dx+\int_{k_{0}}^{k_{0}+1}xf(x)dx+\sum_{k\geq k_{0}+1}\int_{k}^{k+1}xf(x)dx \\
&\geq \sum_{k<0}kf(k+1)+\sum_{k=0}^{k_{0}-1}kf(k)+k_{0}\min \{f(k_{0}),f(k_{0}+1)\}+\sum_{k\geq k_{0}+1}kf(k+1)\\
&\geq \sum_{k\in \mathbb{Z}}kf(k)-k_{0}f(k_{0})-k_{0}f(k_{0}+1)+k_{0}\min\{f(k_{0}),f(k_{0}+1)\}-\sum_{k\in \mathbb{Z}}f(k). 
\end{align*}
Therefore, 
\begin{align*}
\int_{\mathbb{R}} xf(x)dx &\geq \sum_{k\in \mathbb{Z}}kf(k)-\sum_{k\in \mathbb{Z}}f(k)-k_{0}\max f \\
&\geq \sum_{k\in \mathbb{Z}}kf(k)-\sum_{k \in \mathbb{Z}}f(k)-e\sum_{k\in \mathbb{Z}}f(k)=\sum_{k\in \mathbb{Z}}kf(k)-(e+1)\sum_{k\in \mathbb{Z}}f(k) ,
\end{align*}
where in the last line we have used the inequality \eqref{conseqfrad}. Bounding from the above $\int_{\mathbb{R}} xf(x)dx$ we obtain a similar inequality.
Thus, we get
\begin{equation*}
    \left|\int_{\mathbb{R}}xf(x)dx-\sum_{k\in \mathbb{Z}}kf(k)\right|\leq (e+1)\sum_{k\in \mathbb{Z}} f(k) \hspace{0.1cm} ,
\end{equation*}
and the one-dimensional case follows by Proposition \ref{sum-int}.

For the inductive step, let $F:\mathbb{R}^{d-1} \to \mathbb{R}$ be defined, for $x \in \R, y \in \R^{d-1}$, by
\begin{equation*}
F(y):=\int_{\mathbb{R}} f(x,y)dx.
\end{equation*}
Then $F$ is log-concave, centered and isotropic as the marginal of a log-concave, centered and isotropic density and, for $i=1,\ldots,d-1$
\begin{equation*}
    \int_{\mathbb{R}^{d}} y_{i}f(x,y)dxdy=\int_{\mathbb{R}^{d-1}}y_{i}F(y)dy .
\end{equation*}
We have
\begin{align} \nonumber
 &\left| \int_{\mathbb{R}^{d}} y_{i}f(x,y)dxdy - \sum_{k_1\in \Z, k\in \mathbb{Z}^{d-1}} k_{i}f(k_1,k)\right|\\
\nonumber
&\leq \left|\int_{\mathbb{R}^{d}} y_{i}f(x,y)dxdy-\sum_{k\in \mathbb{Z}^{d-1}}k_{i}F(k)\right| + \left|\sum_{k\in \mathbb{Z}^{d-1}}k_{i}F(k)- \sum_{k\in \mathbb{Z}^{d}}k_{i}f(k)\right|.
\end{align}
But the first term is $O_d(1)$ by inductive hypothesis and for the second we note that $f(\cdot,k)$ is log-concave for every $k\in \Z^{d-1}$, so we may apply \eqref{int-sum-dim-1} to get, for any $k\in\Z^{d-1}$, 
\[
\left|F(k) - \sum_{k_1\in\Z}f(k_1,k)\right| =\left| \int_{\R}{f(x,k)dx}-\sum_{k_1\in\Z}{f(k_1,k)}\right|\le\max_{x\in\R}f(x,k),
\]
which gives
\[
\left| \int_{\mathbb{R}^{d}} y_{i}f(x,y)dxdy - \sum_{k_1\in \Z, k\in \mathbb{Z}^{d-1}} k_{i}f(k_1,k)\right| \leq O_d(1)+ \sum_{k\in \mathbb{Z}^{d-1}}k_{i} \max_{x\in  \mathbb{R}} f(x,k) =O_d\bigl({1}\bigr), \label{asinsumofmax}
\]
where  the last equality follows by Lemma \ref{sumofmaximalemma}. 
\end{proof}

We wish to show that a continuous isotropic, log-concave density $f$ in $\mathbb{R}^d$ is almost isotropic in the discrete sense, meaning that its discrete covariance matrix $\Cov_{\mathbb{Z}}{(f)}$ has $O(\sigma)$ off-diagonal elements and $\sigma^2 + O(\sigma)$ diagonal elements.

\begin{proposition} \label{variancediscreteapprox}
Let $f:\mathbb{R}^{d}\rightarrow \mathbb{R}$ be a centered, isotropic, log-concave density with $\Cov{(f)} = \sigma^2\mathrm{I}_{\mathrm{d}}$. Then
\begin{equation*}
    \sum_{k\in\mathbb{Z}^d} f(k)k_i^{2} = \sigma^2 +O_d\big(\sigma\big), \hspace{0.1cm} \text{for every }\ 1\leq i \leq d.
\end{equation*}
\end{proposition}
\begin{proof}

We proceed by induction on the dimension.
For the one-dimensional case let \\ $x_0 = \inf\{x \in \mathbb{R}: f(x) = \max{f}\}$. 
Then
\begin{equation*}
    \int_{\mathbb{R}} f(x)x^{2}=\int_{x\leq x_0}f(x)x^{2}dx+\int_{x>x_0} f(x)x^{2}dx .
\end{equation*}
Noting that $f$ is non-decreasing on $(-\infty,x_0]$ and non-increasing on $[x_0,+\infty)$ by log-concavity, we get
\begin{align} \nonumber
    \int_{\mathbb{R}}f(x)x^{2}dx &=\sum_{k < \fx} \int_{[k,k+1)}f(x)x^{2}dx + \sum_{k > \fx} \int_{[k,k+1)}f(x)x^{2}dx + \int_{\fx}^{\fx+1}f(x)x^2dx
    \\  \nonumber
&\leq \sum_{k < \fx} f(k+1)(k^{2}+k+\frac{1}{3}) + \sum_{k > \fx} f(k)(k^{2}+k+\frac{1}{3}) + \int_{\fx}^{\fx+1}f(x)x^2dx\\ \nonumber
&\leq \sum_{k \in \mathbb{Z}} f(k)k^{2} - \sum_{k < \fx} f(k+1)k  -  \sum_{k > \fx} f(k)k + f(x_0)\Bigl(\fx^2+\fx+\frac{1}{3}\Bigr)  \\ \label{x0lastterm}
&\leq  \sum_{k \in \mathbb{Z}} f(k)k^{2} - \sum_{k \in \mathbb{Z}} f(k)k  + f(x_0)\Bigl(\fx^2+\fx+\frac{1}{3}\Bigr)  
\end{align}
We observe that $f(x_0) = \max{f} \le \frac{1}{\sigma}$ and, by \eqref{maxlog}, $x_0 = O({\sigma})$. Hence, the last term in \eqref{x0lastterm}
is $O(\sigma)$.
Furthermore, by Proposition \ref{sum-int}, the second term in \eqref{x0lastterm} is $O(1)$. Thus, 
$$ \int_{\mathbb{R}}f(x)x^{2}dx \leq \sum_{k \in \mathbb{Z}} f(k)k^{2} + O(\sigma).$$
Similarly, using the piecewise monotonicity of $f$ in the reverse way, we obtain 
$$ \int_{\mathbb{R}}f(x)x^{2}dx \geq \sum_{k \in \mathbb{Z}} f(k)k^{2} + O(\sigma).$$
This proves the base case.\\

Let $x\in \mathbb{R}$, $y\in \mathbb{R}^{d-1}$ and let  
$F(y):=\int_{\mathbb{R}} f(x,y)dx$. Then, by the inductive hypothesis applied to $F$, for $i=1,\ldots,d-1,$
\begin{equation}\label{arg1}
    \int_{\mathbb{R}^{d}}f(x,y)y_{i}^2dxdy=\int_{\mathbb{R}^{d-1}}F(y)y_{i}^2dy=\sum_{k\in \mathbb{Z}^{d-1}}F(k)k_{i}^2+O_d(\sigma) \hspace{0.1cm} .
\end{equation}
Since, for every $k\in\mathbb{Z}^{d-1}$, the function $f(\cdot,k)$ is log-concave in $\mathbb{R}$, we have, by \eqref{int-sum-dim-1}, 
\begin{equation}\label{arg2}
    \left|F(k)-\sum_{k^{\prime}\in \mathbb{Z}} f(k^{\prime},k)\right| = \left|\int_{\mathbb{R}} f(x,k)dx-\sum_{k^{\prime}\in \mathbb{Z}} f(k^{\prime},k)\right| \le \max_{x} f(x,k) \hspace{0.1cm} .
\end{equation}
So by \eqref{arg1} and \eqref{arg2} we get, for $i =1,\ldots,d-1$ 
\begin{align*}
\int_{\mathbb{R}^{d}}f(x,y)y_{i}^2dxdy&=\sum_{ k^{\prime} \in \mathbb{Z},k \in \mathbb{Z}^{d-1} }f(k^{\prime},k)k_{i}^2+\sum_{k\in \mathbb{Z}^{d-1}}O(\max_{x}f(x,k))k_{i}^2+O_d(\sigma) \\
&= \sum_{ k^{\prime} \in \mathbb{Z},k \in \mathbb{Z}^{d-1} }f(k^{\prime},k)k_{i}^2+O_d(\sigma),
\end{align*}
by Lemma \ref{sumofmaximalemma}, and the result follows. 

\end{proof}
\newcommand{\fm}[1]{
 \max_{#1}f(x,y)
}
\newcommand{\re}{\mathbb{R}}
\newcommand{\mm}{
M_{\max}
}

We will need the following lemma for the base case of the induction of Proposition \ref{correlationapprox}. 

\begin{lemma} \label{lemmaintxnmaxf}
Let $f$ be a centered, isotropic, log-concave density in $\mathbb{R}^2$, with $\Cov{(f)} = \sigma^2\mathrm{I}_{\mathrm{2}}$. For $x \in \mathbb{R}$, let $$\nm(x) := \inf\{y \in \mathbb{R}: f(x,y) = \max_y{f(x,y)}\}.$$ Then 
\begin{equation} \label{nmaxint}
    \int_{\mathbb{R}}x\nm(x)\max_yf(x,y)dx = O(\sigma)
\end{equation}
and
\begin{equation} \label{nmaxsum}
    \sum_{k\in \mathbb{Z}}{k\nm(k)\max_yf(k,y)} = O(\sigma).
\end{equation}
\end{lemma}
\begin{proof}
Let $c$ be an absolute constant to be determined later and denote the set $A_{\sigma}:=\{x\in \re: |x|, |\nm(x)| \leq c \sigma\}.$ We have 
\begin{align} \nonumber    \int_{\mathbb{R}}x\nm(x)\max_yf(x,y)dx =  \int_{A_{\sigma}}x\nm(x)\max_yf(x,y)dx + \int_{A_{\sigma}^{\mathrm{C}}}x\nm(x)\max_yf(x,y)dx
\end{align}
The first term can be trivially bounded as
\begin{equation} \nonumber
    \Bigl| \int_{A_{\sigma}}x\nm(x)\max_yf(x,y)dx \Bigr| \leq O(\sigma^3)\max{f} = O(\sigma).
\end{equation}
For the second term we have, for some absolute constant $\cptwo$ and for $c$ large enough,
\begin{align}
   \nonumber
     \Bigl|\int_{A_{\sigma}^{\mathrm{C}}}|x|\nm(x)\max_yf(x,y)dx\Bigr| &\leq \max(f)\int_{A_{\sigma}^{\mathrm{C}}}{x\nm(x)e^{-\frac{|x|+\nm(x)}{\sigma}\cptwo}dx} \\ \label{justifyxnmaxinttrick}
     &\leq O\Bigl(\frac{1}{\sigma^2}\Bigr)\sigma\int_{x>c\sigma}{x e^{-\frac{x}{\sigma}\cptwo}dx}\\
    &= O(\sigma).
\end{align}
Here \eqref{justifyxnmaxinttrick} follows by using the fact that $ye^{-y}\le 1$. This gives \eqref{nmaxint},
\eqref{nmaxsum} follows in a similar way. 
\end{proof}

\begin{proposition} \label{correlationapprox}
Let $f$ be a centered, continuous, isotropic, log-concave density in $\mathbb{R}^d$, with $\Cov{(f)} = \sigma^2\mathrm{I}_{\mathrm{d}}$. Then, for all $i \neq j$
$$
\sum_{k \in \mathbb{Z}^d}f(k)k_ik_j= O_d(\sigma).
$$
\end{proposition}
\begin{proof}
We proceed by induction on $d$. Consider first the base case $d = 2$.
Then, letting $\nm(x) := \inf\{y \in \mathbb{R}: f(x,y) = \max_y{f(x,y)}\}$ and $\mm(y) := \inf\{x \in \re: f(x,y) = \max_x{f(x,y)}\}$ and assuming without loss of generality that the infima are achieved,
\begin{align*}
    \int_{\mathbb{R}_+^2}{f(x,y)xydx} = \int_{\mathbb{R}_+}{x\int_{0\leq y \leq \nm(x)}{f(x,y)ydy} + \int_{y>\nm(x)}{f(x,y)ydy}dx}.
\end{align*}
Observe that
\begin{align} \nonumber
  {\int_{0\leq y \leq \nm(x)}{f(x,y)ydy}} &\leq \sum_{0\leq k \leq \nm(x)}{\int_{[k,k+1)}{f(x,y)ydy}} \\ \label{firstnmax}
  &\leq \sum_{0\leq k \leq \nm(x)-1}{(k+1)f(x,k+1)} + \fm{y}(\nm(x)+1)
\end{align}
and
\begin{align} \label{ygrnmax}
    {\int_{ y > \nm(x)}{f(x,y)ydy}} &\leq \sum_{k\geq\nm(x)}{(k+1)f(x,k)}.
\end{align}
By \eqref{firstnmax} and \eqref{ygrnmax}
\begin{equation} \label{xknm}
 \int_{\mathbb{R}_+^2}{f(x,y)xydx} \leq \int_{\re}{x\Bigl(\sum_{k \geq 0}{kf(x,k)} + \sum_{k \geq \nm(x)}{f(x,k)} + \fm{y}(\nm(x)+1)\Bigr) dx}
\end{equation}
For the first term in \eqref{xknm} we have 
\begin{align} \nonumber
    \int_{\re}{x\sum_{k \geq 0}{kf(x,k)}dx} &\leq \sum_{k \geq 0}{\sum_{0\leq m\leq \mm(k)-1}{k(m+1)f(m+1,k)}} + \sum_{k \geq 0}{\sum_{m > \mm(k)}{k(m+1)f(m,k)}} \\ \nonumber
    &+ \sum_{k \geq 0}{k(\mm(k)+1)\max_xf(x,k)} \\ 
    &\leq \sum_{k\geq 0}\sum_{m \geq 0}{kmf(m,k)} + \sum_{k \geq 0}{\sum_{m \geq 0}{kf(m,k)}} + \sum_{k \geq 0}{k(\mm(k)+1)\max_xf(x,k)} \\ \label{firstterminxknm}
    &\leq \sum_{k\geq 0}\sum_{m \geq 0}{kmf(m,k)} + \sum_{k \geq 0}{\sum_{m \geq 0}{kf(m,k)}} + O(\sigma),
\end{align}
where the last inequality follows by Lemma \ref{lemmaintxnmaxf}.
On the other hand, for the second term in \eqref{xknm} we have
\begin{align} \label{secondterminxknm}
    \int_{\re}{\sum_{k\geq \nm}xf(x,k)dx} \leq \Bigl(\int_{\re}{\sum_{k\in \mathbb{Z}}x^2f(x,k)dx}\Bigr)^{\frac{1}{2}} = O(\sigma) 
\end{align}
by the Cauchy-Schwarz inequality and Proposition \ref{variancediscreteapprox}.
Finally, the third term in \eqref{xknm} is
\begin{equation} \label{thirdterminxknm}
\int_{\re}{\fm{y}(\nm(x)+1)dx} = O(\sigma)
\end{equation}
by Lemma \ref{lemmaintxnmaxf}.
Plugging \eqref{firstterminxknm}, \eqref{secondterminxknm} and \eqref{thirdterminxknm} into \eqref{xknm} we obtain 
\begin{equation} \label{firstquad}
 \int_{\mathbb{R}_+^2}{f(x,y)xydx} \leq \sum_{k\geq 0}\sum_{m \geq 0}{kmf(m,k)} + \sum_{k \geq 0}{\sum_{m \geq 0}{kf(m,k)}} + O(\sigma).
\end{equation}
Using the reverse bounds to those in \eqref{firstnmax} and \eqref{ygrnmax} obtained by the monotonicity of $f$, we get 
the lower bound 
\begin{equation} \nonumber
    \int_{\mathbb{R}_+^2}{f(x,y)xydx} \geq \sum_{k\geq 0}\sum_{m \geq 0}{kmf(m,k)} + \sum_{k \geq 0}{\sum_{m \geq 0}{kf(m,k)}} + O(\sigma).
\end{equation}
In a completely analogous way, we may bound the integral on the rest of the quadrants and obtain the approximations
\begin{align} 
     \int_{\mathbb{R}_-^2}{f(x,y)xydx} &= \sum_{k\leq 0}\sum_{m \leq 0}{kmf(m,k)} + \sum_{k \leq 0}{\sum_{m \leq 0}{kf(m,k)}} + O(\sigma), \\ 
     \int_{\mathbb{R}_-\times \re_+}{f(x,y)xydx} &= \sum_{k\geq 0}\sum_{m \leq 0}{kmf(m,k)} + \sum_{k \geq 0}{\sum_{m \leq 0}{kf(m,k)}} + O(\sigma) \quad \text{and}\\ \label{finalquad}
     \int_{\mathbb{R}_+\times \re_-}{f(x,y)xydx} &= \sum_{k\leq 0}\sum_{m \geq 0}{kmf(m,k)} + \sum_{k \leq 0}{\sum_{m \geq 0}{kf(m,k)}} + O(\sigma).
\end{align}
Since, by Proposition \ref{covdis}, $\sum_{k }{\sum_{m }{kf(m,k)}} = O(1)$, by adding \eqref{firstquad}--\eqref{finalquad} we obtain 
\begin{equation} \nonumber
   0= \int_{\mathbb{R}^2}{f(x,y)xydx} = \sum_{k\in \mathbb{Z}}\sum_{m \in \mathbb{Z}}{kmf(m,k)} + O(\sigma).
\end{equation}
This completes the base case $d = 2$.

The inductive step is similar to that in the proof of Proposition \ref{variancediscreteapprox}. Let $x\in \mathbb{R}$, $y = (y_1,\ldots,y_{d-1})\in \mathbb{R}^{d-1}$ and let  
$F(y):=\int_{\mathbb{R}} f(x,y)dx$. Then, by the inductive hypothesis, for $1\leq i,j \leq d-1,$
\begin{equation}\label{indhyij}
    \int_{\mathbb{R}^{d}}f(x,y)y_{i}y_jdy=\int_{\mathbb{R}^{d-1}}F(y)y_{i}y_jdy=\sum_{k\in \mathbb{Z}^{d-1}}F(k)k_{i}k_j+O_d(\sigma) \hspace{0.1cm} .
\end{equation}
Since for every $k = (k_1,\ldots, k_{d-1})\in\mathbb{Z}^{d-1}$, the function $f(\cdot,k)$ is log-concave $\mathbb{R}$, we have by Proposition \ref{sum-int} 
\begin{equation}\label{indhypintegralapprox}
    F(k)-\sum_{k^{\prime}\in \mathbb{Z}} f(k^{\prime},k) = \int_{\mathbb{R}} f(x,k)dx-\sum_{k^{\prime}\in \mathbb{Z}} f(k^{\prime},k) = O_d\bigl( \max_{x} f(x,k)\bigr) \hspace{0.1cm} .
\end{equation}
So by \eqref{indhyij} and \eqref{indhypintegralapprox} we get, for $1\leq i,j \leq d-1$
\begin{equation} \label{integralcrossstep}
\int_{\mathbb{R}^{d}}f(x,y)y_{i}y_jdxdy=\sum_{ k^{\prime} \in \mathbb{Z},k \in \mathbb{Z}^{d-1} }f(k^{\prime},k)k_{i}k_j+\sum_{k\in \mathbb{Z}^{d-1}}O\Bigl(\max_{x}f(x,k)\Bigr)k_{i}k_j+O_d(\sigma).
\end{equation}
But by the Cauchy-Schwarz inequality and Lemma \ref{sumofmaximalemma}, $\sum_{k\in \mathbb{Z}^{d-1}}O\Bigl(\max_{x}f(x,k)\Bigr)k_{i}k_j=O_d(\sigma)$ and the result follows.
\end{proof}

\begin{corollary} \label{detapprox}
Let $f$ be a centered, isotropic, log-concave density in $\mathbb{R}^d$, with $\Cov{(f)} = \sigma^2\mathrm{I}_d.$ Then
$$
 \det\Bigl(\Cov_{\mathbb{Z}^d}{(f)}\Bigr) =\sigma^{2d}+ O_d(\sigma^{2d-1}).
$$
\end{corollary}
\begin{proof}
By Proposition \ref{correlationapprox}
$$
[\Cov_{\mathbb{Z}^d}(f)]_{ij} = O_d(\sigma), \quad \text{for } i \neq j
$$
and  
$$
[\Cov_{\mathbb{Z}^d}(f)]_{ij} = \sigma^2 + O_d(\sigma), \quad \text{for } i = j.
$$  
\noindent
Thus, 
\begin{align*}
\det\Bigl(\Cov_{\mathbb{Z}^d}{(f)}\Bigr) &= \sum_{\tau}\mathrm{sgn}(\tau)\prod_{i=1}^d\Cov_{\mathbb{Z}^d}(f)_{i\tau(i)}\\
&= \sigma^{2d} + O(\sigma^{2d-1}),
\end{align*}
since for $\tau$ being the identity permutation $\mathrm{sgn}(\tau)\prod_{i=1}^d{\Cov_{\mathbb{Z}^d}{(f)}_{i\tau(i)}} = \bigl(\sigma^{2}+O(\sigma)\bigr)^d = \sigma^{2d}+O(\sigma^{2d-1})$ and for any other permutation $\tau^{\prime}$
$$
\prod_{i=1}^d{\Cov_{\mathbb{Z}^d}{(f)}_{i\tau^{\prime}(i)}} = O((\sigma^2)^{d-2}\sigma^2) = O_d(\sigma^{2d-2}).
$$
\end{proof}

Finally, Corollary \ref{corfisotropic} below is a version of Theorem \ref{discreteUB} with the assumption that the continuous function $f$ is isotropic.
It can  be seen that due to use of Corollary \ref{detapprox}, $\sigma$ needs to be taken at least $\Omega(d!)$. We do not know whether this is the best (lowest) rate. 
\begin{corollary} [{\bf Theorem \ref{discreteUB} for isotropic $f$}]\label{corfisotropic}
Let $p$ be log-concave p.m.f. on $\mathbb{Z}^d$, whose continuous log-concave extension, say $f$, is isotropic. Then there exists an absolute constant $C^{\prime}$ 
such that 
$$
\max_{k \in \mathbb{Z}^d}{p(k)} \leq \frac{C^{\prime}}{\det\Bigl(\Cov(p)\Bigr)^{\frac{1}{2}}}
$$
provided that $\sigma := \det{\Bigl(\Cov_{\mathbb{R}}{(f)}\Bigr)}$ is large enough depending on $d$.
\end{corollary}
\begin{proof}
Since $p$ is extensible log-concave, there exists a continuous log-concave function $f$ (not necessarily a density) such that $f(k) = p(k)$ for all $k\in \mathbb{Z}^d$ and by assumption $f$ is isotropic. Thus 
\begin{equation} \label{Lfbounduse}
\max_{k \in \mathbb{Z}^d}p(k) \leq \max_{x\in \mathbb{R}^d}f(x)=\frac{L_f^d\int_{\mathbb{R}^d} f}{\sigma^d} \leq \frac{C\int_{\mathbb{R}^d}{f}}{\sigma^d} \hspace{0.1cm}, 
\end{equation}
where $C$ is an upper bound of $L_{f}^{d}$, which can be taken to be an absolute constant as guaranteed by the Isotropic Constant Theorem \ref{klartgalehec}. 
Noting that by definition $\Cov{(Zp)} = \Cov{(p)}$ for any normalizing constant $Z \in \R_+$ and applying Proposition \ref{sum-int} and Corollary \ref{detapprox} to  $\frac{f(\cdot)}{\int{f}}$, we obtain
\begin{equation} \nonumber
\max_{k \in \mathbb{Z}^d}p(k) \leq \frac{C(1+O_d(\frac{1}{\sigma}))}{\Bigl(\det\Bigl(\Cov{(p)}\Bigr) + O_d(\sigma^{2d-1})\Bigr)^{1/2}} \leq \frac{2C}{\det\Bigl(\Cov{(p)}\Bigr)^{\frac{1}{2}}}
\end{equation}
provided that $\sigma$ is large enough depending on $d$, using that $O_d(\sigma^{2d-1}) \geq  -\frac{1}{4}\det\Bigl(\Cov{(p)}\Bigr) \simeq -\frac{1}{4}\sigma^{2d}$ as $\sigma \to \infty$. Note that implicitly we also used $\int{f} = O_d(\sum_k{f(k)})$ (for $f$ not necessarily a density) which can be seen directly from the proof of Proposition \ref{sum-int}. The result follows with $C^{\prime} = 2C$.
\end{proof}


\begin{proposition} \label{telescopinglemma}
Fix $d \geq 1$ and let $p$ be a log-concave  p.m.f. on $\mathbb{Z}^d$ with almost isotropic extension. Then, for every $1\le i\le d$, 
\begin{equation*}
   \sum_{k \in \mathbb{Z}^d}|p(k) - p(k-e_i)|=O_d\left(\frac{1}{\det{\left(\Cov{(p)}\right)^{\frac{1}{2d}}}}\right),
\end{equation*}
where  $e_i \in \mathbb{Z}^d$ is defined as the vector with the $i$-th coordinate $1$ and all the other coordinates $0$.
\end{proposition}
\begin{proof}
For any one-dimensional quasi-concave non-negative real sequence $(a_n)_{n\in\Z}$, $\sum_{n\in\Z}|a_n-a_{n-1}|\le2\max_{n\in\Z}a_n$. Thus, by log-concavity, denoting $\overline{k_{i}}=\sum_{j\neq i}k_je_j\in \Z^{d-1}$, one has
\begin{equation*}
    \sum_{k \in \mathbb{Z}^d}|p(k) - p(k-e_i)| \le 2 \sum_{\overline{k_{i}}\in \mathbb{Z}^{d-1}}\max_{k_{i}\in \mathbb{Z}}p(k_1,\ldots,k_{i-1},k_{i},k_{i+1},\ldots,k_d), 
\end{equation*}
Hence, as in \eqref{sumofmaxima}, we have 
\begin{equation*}
\sum_{\overline{k_{i}}\in \mathbb{Z}^{d-1}}\max_{k_{i}\in \mathbb{Z}}p(k_1,\ldots,k_{i-1},k_{i},k_{i+1},\ldots,k_d)=O_d\Bigl(\frac{1}{\det{\bigl(\Cov{(p)}\bigr)^{\frac{1}{2d}}}}\Bigr) 
\end{equation*}
and the result follows.
\end{proof}

We are ready to give the proof of Theorem \ref{logconcavesatisfyTh}:
\begin{proof}[Proof of Theorem \ref{logconcavesatisfyTh}]
Part \ref{logconcave1} follows directly from Corollary \ref{corfisotropic} (in fact with $c_1(d)$ being independent of the dimension $d$). Part \ref{logconcave2} follows from Lemma \ref{LC}, the well-known fact that for a centered log-concave function $f$ on $\R^d$ 
$$f(0) \simeq_d \frac{1}{\det{(\mathrm{Cov}(f))^{\frac{1}{2}}}}$$ and Corollary \ref{detapprox}. Finally, Part \ref{logconcave3} follows directly from Proposition \ref{telescopinglemma}. 
\end{proof}

\subsection{Relaxing the isotropicity assumption on $f$} \label{pisotropicsection}

In this section, we relax the assumptions of Lemma \ref{KB}, so that we obtain conclusions of the same kind under some weaker assumptions on the eigenvalues of the continuous covariance matrix $\Cov_{\mathbb{R}}{(f)}$. Specifically, we will prove:



\begin{proposition} \label{noisot}
For every centered and log-concave density $f:\mathbb{R}^{d}\rightarrow \mathbb{R}$ and for every $\theta\in \mathbb{S}^{d-1}$ one has 
\begin{equation*}\label{relaxin}
c_1^{d+2}f(0)^\frac{1}{d}\sqrt{\lambda_{\min}(\Cov(f))}
 \le \left(\int_{0}^{\infty}r^{d-1}f(r\theta)dr\right)^\frac{1}{d}\le(d+1) c_2^{d+2}f(0)^\frac{1}{d}\sqrt{\lambda_{\max}(\Cov(f))}.
\end{equation*}
where $c_{1}$ and $c_{2}$ are the constants appearing in \eqref{inclusion-d-d+1}.
\end{proposition}
Before giving the proof, we mention the consequences of Proposition \ref{noisot}.
\begin{corollary} \label{cornotiso}
Let $f$ be a a centered, almost isotropic log-concave density. Then for every $\theta\in \mathbb{S}^{d-1}$
\begin{equation*}
 \int_{0}^{\infty}r^{d-1}f(r\theta)dr = \Theta_d(1),
\end{equation*}
as $\sigma := \det{(\Cov{(f)})} \to \infty$.
\end{corollary}
\begin{proof}
Since $f$ is almost isotropic, we have 
$$
\lambda_{\mathrm{min}}\big(\Cov(f)\big) \simeq_d \lambda_{\mathrm{max}}\big(\Cov(f)\big) \simeq_d \sigma^2
$$
as $\sigma \to \infty$. The result follows from Proposition \ref{noisot} since $f(0) \simeq_d \max{f} \simeq_d \frac{1}{\sigma^d}$.
\end{proof}
\begin{remark} \label{remarkrelaxassumption}
The isotropicity assumption was only used in Lemma \ref{KB}, which in turn allowed us to prove the concentration Lemma \ref{LC} which was repeatedly evoked in Section \ref{logconcavesection} to bound the sums of maxima of log-concave functions and thus also the error terms. By Corollary \ref{cornotiso}, a version of Lemma \ref{LC} holds true for almost isotropic densities, although with possibly different constants and for large enough $\sigma$. Therefore, all the previous results of Section \ref{logconcavesection} hold true and thus also the statements of Theorem \ref{EPI} and \ref{diffentropyapprox} under the almost isotropic log-concave assumption.
\end{remark}

\begin{remark}
Note that if the function $f$ is in addition isotropic i.e.  
$\Cov{(f)} = \sigma^2\mathrm{I}_d$, one has that $f(0)\simeq_{d} \frac{1}{\sigma^{d}}$ and so we recover the result of Lemma \ref{KB}. 
\end{remark}




For the above-mentioned proposition, the following lemma which gives bounds for the inradius and circumradius of convex bodies in not necessarily isotropic position will be essential . 
\begin{proof}[Proof of Proposition \ref{noisot}]
Let $f:\mathbb{R}^{d}\rightarrow \mathbb{R}$ be a centered  log-concave density. For simplicity, we denote $\lambda_{\min}=\lambda_{\min}(\Cov(f))$ and analogously for $\lambda_{\max}$. Following the same steps as in the proof of Lemma \ref{KB}, we get that, for any $u\in \mathbb{S}^{d-1}$,
\[
\frac{\lambda_{\min}}{f(0)}\le\int_{K_{d+2}(f)}\langle x,u\rangle^2dx=\frac{1}{f(0)}\int_{\R^d}\langle x,u\rangle^2f(x)dx\le\frac{\lambda_{\max}}{f(0)}.
\]
Using the inclusion relations \eqref{inclusion-d-d+1}, this gives 
\[
c_1^{2d+2}\lambda_{\min}\le\frac{1}{|K_{d+1}(f)|}\int_{K_{d+1}(f)}\langle x,u\rangle^2dx\le c_2^{2d+2}\lambda_{\max}.
\]
Using these inequalities, \eqref{h_K-d-x-u-square} and again the inclusion relations \eqref{inclusion-d-d+1}, we get 
\[
\sqrt{\frac{d+2}{d}}c_1^{d+2} \sqrt{\lambda_{\min}} B_2^d\subset  K_d(f)\subset\sqrt{d(d+2)}c_2^{d+2} \sqrt{\lambda_{\max}}  B_2^d,
\]
which gives inequality \eqref{relaxin}.

\end{proof}

\section{Concluding Remark and Open Question} \label{concludingsection}

For Theorem \ref{logconcavesatisfyTh} and therefore for Theorems \ref{EPI} and \ref{diffentropyapprox} as well, we have assumed that there exists a continuous extension $f$, which is almost isotropic. We have then shown that $f$ is also almost isotropic in the discrete sense. It would be more natural to start with the assumption that the discrete p.m.f. is isotropic. However, in this case we would not be able to use the continuous toolkit to prove our concentration lemma in Section \ref{concentrationlemmasec}. Nevertheless, we suspect that if the discrete p.m.f. is isotropic, then there exists a continuous extension that is almost isotropic and therefore our results for discrete log-concave distributions would still hold under this assumption: 
\begin{question}
Let $X$ be a log-concave random vector with p.m.f.  $p$ on $\mathbb{Z}^d$. 
Assume that $p$ is almost isotropic. Is there a continuous log-concave extension of $p$ on $\mathbb{R}^d$ which is almost isotropic? 
\end{question}

\vspace{2 cm}

\textbf{Data Availability Statement:} 
This study is based on theoretical research and does not involve any empirical data. As such, there are no data sets associated with this manuscript.

\vspace{2 cm}


\bibliographystyle{abbrv}
\bibliography{logconcave} 

\end{document}